\documentclass[10pt,a4paper,twoside]{scrartcl}
\usepackage[utf8]{inputenc}
\usepackage[english]{babel}
\usepackage{lmodern}
\usepackage{amsmath,amssymb,mathtools,amsfonts,amsthm}
\usepackage{graphicx}
\usepackage{xcolor}
\usepackage{subcaption}
\usepackage{paralist}
\usepackage{lscape}
\usepackage{wrapfig}
\usepackage{hyperref}
\usepackage[capitalize, nameinlink, noabbrev]{cleveref}
\crefname{equation}{}{}
\usepackage{autonum}

\usepackage{scrlayer-scrpage}
\ohead{N. Schneider, V. Michel and N. Sneeuw: High-dimensional experiments for the downward continuation using the LRFMP algorithm}

\newtheorem{theorem}{Theorem}[section]
\newcommand{\T}{\mathcal{T}}
\newcommand{\ball}{\mathbb{B}}
\newcommand{\real}{\mathbb{R}}
\newcommand{\nat}{\mathbb{N}}
\newcommand{\dic}{\mathcal{D}}

\newcommand{\RFMP}{\mathrm{RFMP}}
\newcommand{\N}{\mathcal{N}}

\newcommand{\SH}{\mathrm{SH}}
\newcommand{\APK}{\mathrm{APK}}
\newcommand{\APW}{\mathrm{APW}}

\newcommand{\Hs}[1]{\mathcal{H}_#1(\Omega)}
\newcommand{\Lp}[1]{\mathrm{L}^#1(\Omega)}
\newcommand{\lon}{\varphi}
\newcommand{\lat}{\theta}

\newcommand{\era}{\varepsilon^r}
\newcommand{\ephi}{\varepsilon^\lon}
\newcommand{\ete}{\varepsilon^t}

\newcommand{\trans}{\mathrm{T}}

\newcommand{\intd}{\mathrm{d}}
\newcommand{\pdervr}{\frac{\partial}{\partial r}}
\newcommand{\pdervlon}{\frac{\partial}{\partial \lon}}
\newcommand{\pdervt}{\frac{\partial}{\partial t}}

\title{High-dimensional experiments \\ for the downward continuation \\ using the LRFMP algorithm}
\author{N. Schneider$^\ast$, V. Michel\footnote{Geomathematics Group Siegen, University of Siegen, michel@mathematik.uni-siegen.de, naomi.schneider@mathematik.uni-siegen.de}\ \phantom{i} and N. Sneeuw\footnote{Institute of Geodesy, University of Stuttgart}}
\date{}

\usepackage{natbib}
\begin{document}
\maketitle

\begin{abstract}
Time-dependent gravity data from satellite missions like GRACE-FO reveal mass redistribution in the system Earth at various time scales: long-term climate change signals, inter-annual phenomena like El Ni\~{n}o, seasonal mass transports and transients, e.\,g. due to earthquakes. For this contemporary issue, a classical inverse problem has to be considered: the gravitational potential has to be modelled on the Earth's surface from measurements in space. This is also known as the downward continuation problem. Thus, it is important to further develop current mathematical methods for such inverse problems. For this, the (Learning) Inverse Problem Matching Pursuits ((L)IPMPs) have been developed within the last decade. Their unique feature is the combination of local as well as global trial functions in the approximative solution of an inverse problem such as the downward continuation of the gravitational potential. In this way, they harmonize the ideas of a traditional spherical harmonic ansatz and the radial basis function approach. Previous publications on these methods showed proofs of concept. Here, we consider the methods for high-dimensional experiments settings with more than 500\,000 grid points which yields a resolution of 20\,km at best on a realistic satellite geometry. We also explain the changes in the methods that had to be done to work with such a large amount of data. The corresponding code (updated for big data use) is available at \textcolor{blue}{\protect{\url{https://doi.org/10.5281/zenodo.8223771}}} under the licence CC BY-NC-SA 3.0 Germany.
\end{abstract}



\paragraph{Acknowledgments}
The authors gratefully acknowledge the financial support by the German Research Foundation (DFG; Deutsche Forschungsgemeinschaft), project MI 655/14-1. Further, we are grateful for using the HPC Clusters Horus and Omni maintained by the ZIMT of the University of Siegen for our numerical results. Last but not least, we thank the HPC support team, in particular Dipl.-Inf. Monika Harlacher, for their suggestions on code optimization.

\section{Introduction}
\label{sect:intro}
Time-dependent gravity data from satellite missions like GRACE and GRACE-FO are commonly used in order to monitor the mass transport of the Earth, see e.\,g.\ \cite{Chenetal2022,Fischeretal2013-2,IPCC2023,Landereretal2020,NasaConsensu,Saemianetal2022,Tapleyetal2019,Wieseetal2022} as well as the results of the DFG SPP 1257 (2006-2014) coordinated by Ilk and Kusche, see e.\,g. \cite{Kuscheetal2012}. For literature regarding GRACE and GRACE-FO data, see e.\,g.  \cite{Devaraju2017,Flechtneretal2014,Flechtneretal2014-2,GRACEdata,Schmidtetal2008,Tapleyetal2004,GRACEdata2}. This enables the observation of climate change seen in melting glaciers but also droughts and seasonal phenomena like the wet season in the Amazon basin. These insights are obtained from time variations relative to a temporal mean. 

Modelling the gravitational potential means solving the inverse problem of the downward continuation, see e.\,g. \cite[chapters 3 and 5]{Freedenetal2004} or \cite{Baur2014,Moritz2010}:
\begin{align}
V\left(\sigma\eta\right) 
= \left( \T f \right) \left(\sigma\eta\right) 
= \sum_{n=0}^\infty \sum_{j=-n}^n f_{n,j}\sigma^{-n-1}Y_{n,j}\left(\eta\right),
\label{eq:gravpotatorbit}
\end{align}
which holds pointwise and is represented here for a spherical Earth with radius 1. Note that here $f$ is the desired potential on the Earth's surface and $V$ is the measured potential on a satellite orbit. The expansion here uses spherical harmonics $Y_{n,j}$ and the base of the singular values $\sigma > 1$ is the ratio of the satellite orbit and the Earth's radius. The coefficients $f_{n,j}$ refer to the expansion of $f$ at the Earth's surface. Due to the exponentially decreasing singular values of the upward continuation operator $\T$, this is by nature an exponentially ill-posed inverse problem. For details on inverse problems, see \cite{Engletal1996,Louis1989,Michel2005,Michel2020,Rieder2003}. 

Ill-posed inverse problems can be tackled with different approaches. Most of them have in common that the solution is approximated by a specific, a-priori chosen basis. For GRACE and GRACE-FO data, these are typically either spherical harmonics, radial basis functions or mascons. Spherical harmonics represent the traditional approach, see e.\,g.\, \cite{Chenetal2021}. However, it produced a North-South striping and sophisticated methods for a destriping were accompanied by a signal loss, see e.\,g.\, \cite{Chenetal2021,Watkinsetal2015}. Higher resolutions without the need for destriping methods were obtained by using mascons, see e.\,g.\, \cite{Luthckeetal2013,Saveetal2016,Watkinsetal2015}. 

Spherical harmonics are polynomials and, thus, are perfectly localized in the frequency but not the spatial domain. Mascons, however, are perfectly spatially bounded trial functions and, thus, are not localized in the frequency domain. In this respect, they represent two opposing types of trial functions from the broad range that is known in mathematics. As usual, both of them have by construction benefits and disadvantages (beyond the mentioned GRACE-related destriping aspects): with spherical harmonics, local errors oscillate globally while mascons are prone to overfitting if their underlying geometry is chosen sub-optimal, see \cite{Watkinsetal2015}. Moreover, recently, \cite{Guentneretal2023} showed that the calculation of water mass loss out of GRACE and GRACE-FO observations can notably differ depending on the chosen model. In this particular case, this might be attributable to the selection between spherical harmonics and mascons as trial functions. Radial basis functions yield a trade-off between the extremals of orthogonal polynomials on the one hand and mascons / a dirac Delta on the other hand. Their scaling parameter can additionally be used to balance spatial and spectral localization according to the specific needs, see e.\,g.\,\cite{Freedenetal1999,Freedenetal2018}.

Due to the drawbacks of a-priori choosing the basis, an alternative was implemented in terms of the LIPMP algorithms, see \cite{Fischeretal2013-1,Kontaketal2018-1,Michel2015-2,Micheletal2017-1,Micheletal2016-1,Schneideretal2022}. In these methods, the signal, e.g. the gravitational potential $f$, given by values of $V = \T f$ at discrete points, is approximated by a linear combination of weighted dictionary elements. The corresponding dictionary is a set of diverse trial functions, it often contains global as well as local ones. In the case of the downward continuation, we propose to use spherical harmonics as well as radial basis functions and wavelets. In particular, we use the Abel--Poisson kernel and wavelet. In comparison to mascons, which are similar to finite elements, these have certain mathematical benefits such as a closed form. Note that they can be defined by spherical harmonics in a series representation. 

The approximation obtained from an LIPMP algorithm is usually a combination of different basis functions. In this respect, we try to harmonize the extremes of an ansatz using either only global or only local trial functions. The dictionary elements of our approximation are chosen iteratively, that means, in each iteration, we choose one further dictionary element to be added to the linear combination. The dictionary element is chosen such that the Tikhonov functional is minimized. Due to the learning add-on, the dictionary in use is infinite as we allow the methods to choose a radial basis function and wavelet to be freely picked from the infinite range of their characteristic parameters. Note that this remedies the need to a-priori choose a finite set of trial functions which is the equivalent struggle to determining a suitable mascon geometry, see e.\,g.\, \cite{Watkinsetal2015}.

As we saw in \cite{Micheletal2018-1,Schneider2020,Schneideretal2022}, the learning add-on reduces the computational costs enormously. Thus, it enables us to consider the use of more data. While we showed in the mentioned literature that the LIPMPs can work with regular and irregular data grids, we here did not only increase the number of grid points but also allowed a realistic observing geometry as dictated by the GRACE orbits. In particular, we used the kinematic orbits of the Institute of Geodesy at the TU Graz (\cite{Suesser-Rechberger2022}) of December 2022 from GRACE-FO 1 and 2 such that we now include slightly more than 500\,000 grid points in our methods. We demonstrate here that we are able to work with a sensible amount of data on a realistic distribution. Thus, in our experiments, we will approximate the full potential and not monthly deviations as we did in \cite{Schneider2020,Schneideretal2022}. 

However, this demanded a revision of the implementation of the methods on different levels. First, we improved  the efficiency of the source code. Moreover, we discovered that the computation of certain inner products of two radial basis functions and / or wavelets consumed a large amount of the overall runtime. Previously, we used the Clenshaw algorithm for it which is usually suitable for such tasks. However, due to the learning add-on, the number of calls increased massively.  We derived a closed form for this kind of inner products which is naturally much more efficient and more accurate than any computation of a truncation. At last, in \cite{Schneider2020,Schneideretal2022}, we included Slepian functions in the (L)IPMPs for the first time. In the aftermath, we recognized that including them increased the runtime massively. Thus, we concluded that it seems more sensible to abstain from using this type of trial functions (and risk a few more iterations doing so). Similarly, in our experiments here, we also abstain from the orthogonalization routine in the Learning Regularized Orthogonal Functional Matching Pursuit (LROFMP) and only work with the Learning Regularized Functional Matching Pursuit (LRFMP). Note, however, that the corresponding source code to this paper contains all (L)IPMPs, that is, it also includes the (L)ROFMP in the currently most efficient variant.

We have already introduced the basic geodetic aspects for the downward continuation of the gravitational potential from a satellite orbit above. Next, we present our choices of dictionary elements in \cref{ssect:basics:tfcs} and the method in \cref{ssect:basics:lrfmp}. At last, we present our numerical results in \cref{sect:numerics}. In \cref{sect:app1}, the interested reader finds the proofs for the closed forms of the inner products of two filters. 

\subsection{Notation}
\label{ssect:intro:nota}
The set of real numbers is denoted by $\real$. Moreover, let $\real^d$ be the corresponding real, $d$-dimensional vector space. The set of positive integers is denoted by $\nat$ and $\nat_0$ if we include 0 as well. The unit sphere is abbreviated by $\Omega\coloneqq \{x \in \real^3 :\ |x|=1\}$ and the open unit ball is given by $\ball \coloneqq \ball^3 \coloneqq \{x \in \real^3 :\ |x|< 1\}$. Recall that, for $\eta(\lon,t) \in \Omega$, we can use spherical coordinates: the longitude $\lon \in [0,2\pi[$ and the polar distance $t = \cos(\lat) \in [-1,1]$ for the co-latitude $\lat \in [0,\pi]$. That is, we have 
\begin{align}
\Omega \ni \eta(\lon,t) \coloneqq \left( \begin{matrix}
\sqrt{1-t^2}\cos(\lon)\\
\sqrt{1-t^2}\sin(\lon)\\
t
\end{matrix} \right).
\end{align}

\section{Mathematical basics}
\label{sect:basics}

Next we summarize necessary mathematical basics. In particular, we will define the trial functions under consideration, the approximation methods and the aspects we developed for high-dimensional experiments.

\subsection{Trial functions}
\label{ssect:basics:tfcs}

The trial functions we consider as dictionary elements for the downward continuation of the gravitational potential from satellite data are: spherical harmonics (SHs) as well as Abel-Poisson kernels (APKs) and wavelets (APWs). The kernels serve as radial basis functions and are low-pass filters. The wavelets act as the respective radial basis wavelets and are band-pass filters. Thus, we include global and local functions in the dictionary. Examples are given in \cref{fig:tfcs}.

As global polynomials on $\Omega$ we use SHs $Y_{n,j}$. They are characterized by a discrete degree $n \in \nat_0$ and order $j=-n,...,n$. An example often used in practice is given by the fully normalized SHs 
\begin{align}
Y_{n,j}(\eta(\lon,t)) \coloneqq p_{n,j}P_{n,|j|}(t) \left\{ \begin{matrix} \cos(j\lon),&j\leq 0\\ \sin(j\lon),&j> 0 \end{matrix} \right. \label{def:fnsh}
\end{align}
where $\eta(\lon,t) \in \Omega$, $p_{n,j}$ is an $\Lp{2}$-normalization and $P_{n,|j|}$ is an associated Legendre function with $P_n \coloneqq P_{n,0}$ being the Legendre polynomial. Further details can be found, e.\,g., in \cite[chapter 3]{Freedenetal1998}, \cite[chapter 3]{Freedenetal2009} and \cite{Mueller1966}.

APKs $K(x,\cdot)$ and APWs $W(x,\cdot)$ are localized radial basis functions. They come in handy in practice due to their closed form (see e.g. \cite[Example 6.23]{Michel2013}):
\begin{align}
K(x,\eta) &\coloneqq \frac{1-|x|^2}{4\pi(1+|x|^2-2x\cdot\eta)^{3/2}} = \sum_{n=0}^\infty \frac{2n+1}{4\pi} |x|^n P_n \left(\frac{x}{|x|}\cdot\eta\right)\label{def:apk}
\intertext{and}
W(x,\eta) &\coloneqq K(x,\eta) - K(|x|x,\eta)	\label{def:apw}
\end{align}
for $\eta \in \Omega$. Obviously, they are characterized by $x\in\ball$. Further details on radial basis functions can be found, e.\,g., in  \cite[chapter 5]{Freedenetal1998}, \cite[section 2.3]{Freedenetal2004}, \cite{Freedenetal1998-1}, \cite[chapter 7 and 10]{Freedenetal2009}, \cite{Freedenetal1996}, \cite[chapter 6]{Michel2013} and \cite{Windheuser1995}.
The dictionary is then defined as follows: 
\begin{align}
[N]_{\SH} &\coloneqq \left\{Y_{n,j}(\cdot)\ \colon\ (n,j) \in N \subseteq \N \right\}, \label{def:tfc:sh} \qquad
\N \coloneqq \left\{(n,j)\ \colon\ n\in \nat_0,\ j=-n,...,n\right\}\\
[B_K]_{\APK} &\coloneqq \left\{ K(x,\cdot)\ \colon\ x \in B_K \subseteq \ball \right\}	\label{def:tfc:apk}\\
[B_W]_{\APW} &\coloneqq \left\{ W(x,\cdot)\ \colon\  x \in B_W \subseteq \ball \right\} \label{def:tfc:apw}\\
\dic &\coloneqq [N]_{\SH} \cup [B_K]_{\APK} \cup [B_W]_{\APW}.\label{def:dic}
\end{align}
Please refer to \cite{Michel2020,Micheletal2018-1,Schneider2020,Schneideretal2022} for details on a general dictionary definition.

\begin{figure}
\centering
\includegraphics[width=.32\textwidth]{"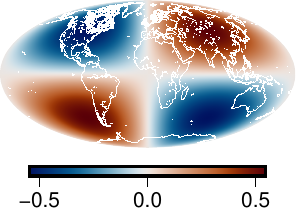"}\\[2\baselineskip]
\includegraphics[width=.49\textwidth]{"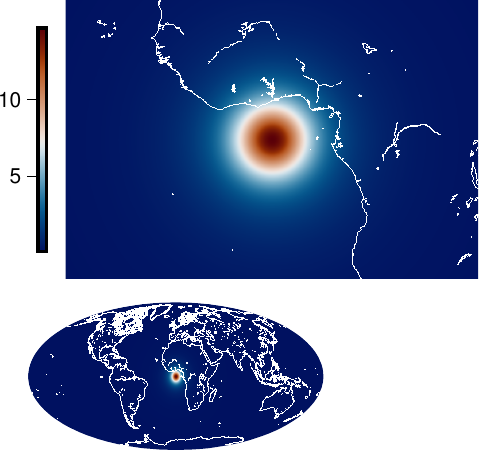"}
\includegraphics[width=.49\textwidth]{"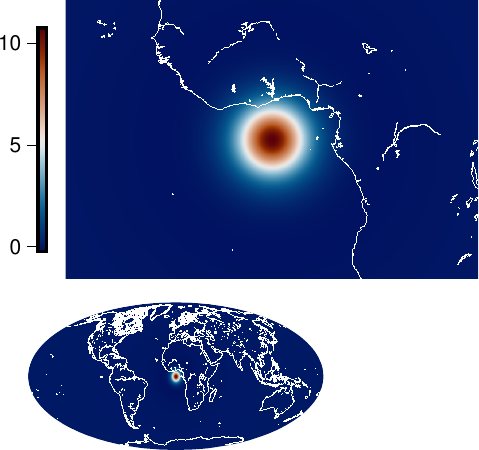"}
\caption{Examples of dictionary elements. Upper row: Fully normalized spherical harmonic $Y_{2,1}$. Lower row: Abel--Poisson kernel $K(x,\cdot)$ (left) and Abel--Poisson wavelet $W(x,\cdot)$ (right) both for $x=(0.9,0,0)^\trans$ on a global scale and zoomed-in on their extremum.}
\label{fig:tfcs}
\end{figure}

\subsection{The LRFMP}
\label{ssect:basics:lrfmp}

A wide a range of literature exists on the LIPMP algorithms and their applicability, see \cite{Fischeretal2013-1,Kontaketal2018-1,Michel2015-2,Micheletal2017-1,Micheletal2016-1,Prakashetal2020,Telschowetal2018,Schneideretal2022} and the references therein. In our numerical experiments here, we will only use the LRFMP. Thus, we summarize this method next and refer to the mentioned literature for more insights.

\subsubsection{Main approach}
\label{sssect:basics:lrfmp:general}

First we need a mathematical description of the problem at hand. Let the unit sphere $\Omega$ denote the Earth's surface. Then the downward continuation of the gravitational potential from satellite data to the Earth's surface is described by $y=\T_\daleth f$ with the data $y \in \real^\ell,\ y_i = V(\sigma\eta^i)$, the satellite height $\sigma >1$, grid points $\eta^i \in \Omega$ for $i=1,..., \ell$ at the Earth's surface and the discretization operator $\T_\daleth f \coloneqq ((\T f)(\sigma\eta^i))_{i=1,...,\ell}$ with the upward continuation operator $\T$ as in \cref{eq:gravpotatorbit}. We desire an approximation of $f$. As we have seen in \cref{sect:intro}, this is an ill-posed inverse problem which makes a regularization inevitable. The LRFMP implements a strategy for minimizing the Tikhonov-Phillips functional. This minimization is obtained in an iterative process: starting with an initial (guessed) approximation $f_0$, e.\ g. $f_0 \equiv 0$, one additional weighted dictionary element $d_n \in \dic$ is chosen in each iteration:
\begin{align}
f_N &= f_0 + \sum_{n=1}^N \alpha_n d_n. \label{eq:fn}
\end{align}
As we have seen that our dictionary consists of diverse trial functions, the approximation $f_N$ will usually be built in a heterogeneous basis. This is intentional as it allows to combine all benefits from each type of trial function while reducing possible repercussions (such as global effects due to local anomalies). Besides the approximation, we have in each iteration a residual:
\begin{align}
R^{N+1} &\coloneqq R^N - \alpha_{N+1}\T_\daleth d_{N+1} \label{def:RN}
\end{align}
with $R^0 = y-\T_\daleth f_0$. In each iteration $N$, the weights $\alpha_{N+1} \in \real$ and the basis function $d_{N+1} \in \dic$ shall be chosen such that the Tikhonov-Phillips functional is minimized. That is, for $\lambda>0$, we minimize
\begin{align}
(\alpha, d) \mapsto \left\| R^N - \alpha \T_\daleth d \right\|^2_{\real^\ell} + \lambda\left\|f_N+\alpha d\right\|^2_{\Hs{2}}	\label{eq:TFNO}
\end{align}
From experience, we choose the Sobolev space $\Hs{2} \subset \Lp{2}$ for the penalty term. This Sobolev space is defined as the completion of the set of all square-integrable functions with
\begin{align}
\|F\|^2_{\Hs{2}} \coloneqq \sum_{n=0}^\infty \sum_{j=-n}^n (n+0.5)^4 \langle F, Y_{n,j} \rangle^2_{\Lp{2}} < \infty, 
\end{align}
see e.\,g. \cite[pp.\ 82--83]{Freedenetal1998} and \cite[pp.\ 150--151]{Michel2013}. The minimization of the Tikhonov-Phillips functional is, in practice, exchanged by the equivalent maximization of
\begin{align}
\RFMP(d;N) &\coloneqq \frac{\left( \left\langle R^N, \T_\daleth d\right\rangle_{\real^\ell} - \lambda\left\langle f_N,d \right\rangle_{\Hs{2}} \right)^2}{\left\|\T_\daleth d\right\|_{\real^\ell}^2 + \lambda\left\|d\right\|^2_{\Hs{2}}}
\eqqcolon \frac{\left(A_N(d)\right)^2}{B_N(d)}.\label{def:RFMP(d;N)}
\end{align}
The weights are then straightforwardly obtained by 
\begin{align}
\alpha_{N+1} &\coloneqq \frac{A_N(d_{N+1})}{B_N(d_{N+1})}.	\label{def:alphan}
\end{align}
This maximization is an optimization problem in a function space of diverse trial functions. Thus we propose to first maximize \cref{def:RFMP(d;N)} for each type of trial functions separately. This yields again a finite dictionary of (optimized) candidates (at least one per type of trial function) from which we can choose the overall maximizer by pairwise comparing the values for the candidates in \cref{def:RFMP(d;N)}. The optimization for each type of trial functions can be done in two ways, depending on whether their characteristic parameters are discrete (as in the case of the SHs) or continuous (as in the case of the APKs and APWs). In the former case, we do a pairwise comparison of the objective function \cref{def:RFMP(d;N)} up to a large, a-priori fixed maximal degree. In the latter case, we obtain a true non-linear constrained optimization problem with respect to the characteristic parameters for which we can utilize established optimization routines. Note that in previous publications the determination of the dictionary of candidates was also described as a learning add-on. 

By construction, we need to define termination criteria for the LRFMP. Usually, the methods terminate if the relative data error falls below a certain threshold like the noise level (see also the Morozov discrepancy principle, \cite{Morozov1966}) or if a certain number of iterations is reached. 

\subsubsection{Aspects for upscaling}
\label{sssect:basics:lrfmp:novel}
The previous publications on the (L)RFMP, \cite[section 5.4]{Michel2020}, \cite{Micheletal2018-1,Schneider2020,Schneideretal2022}, showed tests with up to $\approx$ 12\,500 data points. This is by far not enough for the following reasons: on the one hand, each satellite of the GRACE-FO mission does $\approx$ 8\,500 measurements each day; on the other hand, the EGM2008 is a model of the gravitational potential in spherical harmonics up to degree 2190 and order 2159. That means, current models usually use much more data which is also available. As we have seen in the mentioned literature, the computational cost could be massively reduced by using the LRFMP instead of a non-learning IPMP variant. However, this was not the case in all experiments and the remaining runtime costs were still too expensive for high-dimensional data usage. Thus, we did a revision of the methods, the settings and the implementation in order to increase the number of used data. 

First of all, in the first publications on the LIPMP, we included Slepian functions in the dictionary as well. Slepian functions are band-limited weighted linear combinations of spherical harmonics which are (thus) perfectly localized in the frequency and optimally localized in space, see \cite{Albertellaetal1999,Simonsetal2006}. We originally included Slepian functions hoping that it would reduce the number of iterations and, thus, the size of the best basis. However, we saw that learning Slepian functions made up a large amount of the runtime and it appears that the cost of learning Slepian functions is not outweighed by less iterations. The large numerical effort was caused by the optimization routine as part of the learning process. Thus, we abstain from learning Slepian functions for now. Possibly, future research might find a way for a more efficient learning.

A first implementation was improved by certain optimization hints based on \cite{Hageretal2011}. We also tested whether our implementation would benefit if we switched from C/C++ to the high-level programming language Julia. However, this was, in our case, not very fruitful such that we stick to the C/C++ source code.

Besides those optimization aspects, we also checked the runtime with gprof, see e.g. \cite{GNUgprof}. This is a well-established open-source profiling tool for runtime analysis. This drew our attention to the Sobolev inner product of two radial basis functions or wavelets or a mixture of one radial basis function and one radial basis wavelet. In the objective function \cref{def:RFMP(d;N)}, we use the Sobolev norm of the current approximation $f_N$ and a dictionary element $d$ in the numerator and the Sobolev norm of this dictionary element in the denominator. As we have $f_N = f_0 + \sum_{n=1}^N \alpha_n d_n$, the former one also depends only on the inner product of two dictionary elements $\langle d,\tilde{d}\rangle_{\Hs{2}}$. In the case that $d,\tilde{d} \in [B_K]_{\APK}\cup [B_W]_{\APW}$, the inner product is a series of weighted Legendre polynomials, see e.g. \cite{Schneider2020}. Traditionally, such a series is approximated by a truncation where the latter can be efficiently computed with the Clenshaw algorithm, see e.g. \cite{Deuflhard1976}. As these terms are part of the objective function \cref{def:RFMP(d;N)}, the algorithm has to be called in each function evaluation within the learning optimization of the kernels and wavelets in each iteration of the LRFMP. Moreover, the values obviously cannot be preprocessed due to the learning optimization. With these many calls of the Clenshaw algorithm, however, it cannot be most efficient anymore. 

As the truncation of the obtained series is a question of accuracy anyway, we are concerned with finding a closed form for the Sobolev inner products
\begin{align}
\langle K(x,\cdot),K(\tilde{x},\cdot)\rangle_{\Hs{2}} &= \sum_{n=0}^\infty (n+0.5)^4 \left(|x|\left|\tilde{x}\right|\right)^n \frac{2n+1}{4\pi} P_n\left(\frac{x}{|x|}\cdot \frac{\tilde{x}}{|\tilde{x}|}\right)  \\
\langle K(x,\cdot),W(\tilde{x},\cdot)\rangle_{\Hs{2}} &= \sum_{n=0}^\infty (n+0.5)^4 \left(|\tilde{x}|^n - |\tilde{x}|^{2n}\right)\left|x\right|^n \frac{2n+1}{4\pi} P_n\left(\frac{x}{|x|}\cdot \frac{\tilde{x}}{|\tilde{x}|}\right)\\
\langle W(x,\cdot),W(\tilde{x},\cdot)\rangle_{\Hs{2}} &= \sum_{n=0}^\infty (n+0.5)^4 \left(|x|^n - |x|^{2n}\right) \left(\left|\tilde{x}\right|^n - \left|\tilde{x}\right|^{2n}\right) \\&\qquad\qquad\qquad\qquad\qquad\qquad\qquad \times\frac{2n+1}{4\pi} P_n\left(\frac{x}{|x|}\cdot \frac{\tilde{x}}{|\tilde{x}|}\right),
\end{align}
see e.g. \cite[section 3.5]{Schneider2020}. Obviously, these inner products are all sums of terms of the form
\begin{align}
\frac{1}{4\pi} \sum_{n=0}^\infty (n+0.5)^4(2n+1)(h^{m}\tilde{h}^{\tilde{m}})^n P_n\left(\frac{x}{|x|}\cdot \frac{\tilde{x}}{|\tilde{x}|}\right) 
\eqqcolon 
\frac{1}{4\pi} \sum_{n=0}^\infty (n+0.5)^4(2n+1)q^n P_n(\tau),\label{eq:commonterminIPs}
\end{align}
where $h=|x|,\ \tilde{h}=|\tilde{x}| \in [0,1[,\ m,\ \tilde{m} \in \{1,2\}$. Note that $m$ and $\tilde{m}$ depend on the considered dictonary element: whether it is an Abel-Poisson kernel or wavelet. Further note that, if we want to use a gradient-based optimization method in the learning process, we also need to compute the gradient of \cref{eq:commonterminIPs} with respect to $x= (h/|x|)x$. We immediately obtain for \cref{eq:commonterminIPs}
\begin{align}
\frac{1}{4\pi} \sum_{n=0}^\infty (n+0.5)^4(2n+1)q^n P_n(\tau)
&= \frac{1}{64\pi} \sum_{n=0}^\infty (2n+1)^5q^n P_n(\tau)\\
&= \frac{1}{64\pi} \sum_{n=0}^\infty \sum_{k=0}^5 \binom{5}{k}(2n)^k q^n P_n(\tau)\\
&= \frac{1}{64\pi} \sum_{k=0}^5 \binom{5}{k}2^k \sum_{n=0}^\infty n^k q^n P_n(\tau),
\end{align}
where the latter formulation is possible due to the absolute convergence of the series. Thus, we are left with finding a closed form for 
\begin{align}
\sum_{n=0}^\infty n^k q^n P_n(\tau).
\end{align}
It is well-known, see for instance \cite[equation (5.6.17)]{Freedenetal1998} and \cite[Example 6.23]{Michel2013}, that
\begin{align}
\sum_{n=0}^\infty q^n P_n(\tau) = \frac{1}{\sqrt{1 + q^2 - 2q\tau}} \eqqcolon \phi(q).
\label{eq:genfunAPK}
\end{align}
Note that, in Cartesian coordinates, we have
\begin{align}
\phi(q) = \phi(|x|) 
&= \left(1+|x|^2 - 2|x|\frac{x}{|x|}\cdot\frac{\tilde{x}}{|\tilde{x}|}\right)^{-1/2}
= \left(1+|x|^2 - 2x\cdot\frac{\tilde{x}}{|\tilde{x}|}\right)^{-1/2}\\
&= \left(|\tilde{\eta}|^2 +|x|^2 - 2x\cdot\tilde{\eta}\right)^{-1/2}
= \left(|\tilde{\eta}-x|^2\right)^{-1/2}
= \frac{1}{|\tilde{\eta}-x|},
\end{align}
which is well-known to be infinitely often continuously differentiable if $x\not=\eta$. The latter is true because $\eta\in\Omega$ and $|x|=q\in[0,1[$. Note that a spherical gradient is given by
\begin{align}
\nabla &= \era(\lon,t) \pdervr + \frac{1}{r} \left( \ephi(\lon,t) \frac{1}{\sqrt{1-t^2}}\pdervlon + \ete(\lon,t) \sqrt{1-t^2}\pdervt \right),
\end{align}
see e.g. \cite[Theorem 4.5]{Michel2013}. Then, we obviously have
\begin{align}
q\intd_q = |x| \pdervr = x\cdot\nabla_x,
\end{align}
which is equally well-defined. Hence, with \cref{eq:genfunAPK}, we obtain the following theorem.
\begin{theorem}
Let $k\in \nat_0, q\in [0,1[$ and $\tau \in [-1,1]$. Then it holds
\begin{align}
\sum_{n=0}^\infty n^k q^n P_n(\tau) 
= \underbrace{\left( q \frac{\intd}{\intd q} \right)\cdots\left( q \frac{\intd}{\intd q} \right)}_{k-\mathrm{times}} \phi(q)
\eqqcolon \left( q \frac{\intd}{\intd q} \right)^k \phi(q)
\eqqcolon \left( q \intd_q \right)^k \phi(q).
\end{align}
\end{theorem}
\begin{proof}
We prove the theorem by induction with respect to $k$. Obviously, the case $k=0$ is given by \cref{eq:genfunAPK}. Hence, we consider $k$ to $k+1$. We start at the right-hand side:
\begin{align}
\left( q \intd_q \right)^{k+1} \phi(q)
&= q \intd_q \left( q \intd_q \right)^k \phi(q)\\
&= q \intd_q \sum_{n=0}^\infty n^k q^n P_n(\tau) 
= q \sum_{n=1}^\infty n^{k+1} q^{n-1} P_n(\tau) 
= \sum_{n=0}^\infty n^{k+1} q^n P_n(\tau),
\end{align}
where we used $|P_n(\tau)| \leq 1$ for all $\tau\in[-1,1]$ and the uniform convergence of the differentiated series for $0\leq q \leq q_0<1$ with a fixed $q_0$. Since, however, this conclusion can be done for all $q_0\in\ ]0,1[$, we obtain the result for all $q\in [0,1[$.
\end{proof}
Thus, we obtain for \cref{eq:commonterminIPs}: 
\begin{align}
\frac{1}{4\pi} &\sum_{n=0}^\infty (n+0.5)^4(2n+1)q^n P_n(\tau)\\
&= \frac{1}{64\pi} \sum_{k=0}^5 \binom{5}{k}2^k \left( q \intd_q \right)^k \phi(q)\\
&= \frac{1}{64\pi}\left(\phi(q) + 10\left( q \intd_q \right)\phi(q) + 40\left( q \intd_q \right)^2\phi(q) + 80\left( q \intd_q \right)^3\phi(q) \right.\\
&\qquad \qquad \qquad \left. +\ 80\left( q \intd_q \right)^4\phi(q) + 32\left( q \intd_q \right)^5\phi(q) \right).
\label{eq:closedformSobIPs}
\end{align}
This leaves us with the determination of $\left( q \intd_q \right)^k \phi(q),\ k=0,...,5,$ for practical purposes. We state here the results. The computation can be found in \cref{ssect:app1:derivsubterms}.
\begin{align}
\phi(q) 
&= \frac{1}{(1+q^2-2q\tau)^{1/2}}, \label{eq:qdq0phi}\\
q \intd_q \phi(q) 
&= \frac{\tau q-q^2}{(1+q^2-2q\tau)^{3/2}},\label{eq:qdq1phi}\\
\left(q \intd_q\right)^2 \phi(q) 
&= \frac{\tau q-2q^2}{(1+q^2-2q\tau)^{3/2}} + \frac{3(\tau q-q^2)^2}{(1+q^2-2q\tau)^{5/2}}, \label{eq:qdq2phi}\\
\left(q \intd_q\right)^3 \phi(q)
&= \frac{\tau q-4q^2}{(1+q^2-2q\tau)^{3/2}} + \frac{9(\tau q-q^2)(\tau q-2q^2)}{(1+q^2-2q\tau)^{5/2}} + \frac{15(\tau q-q^2)^3}{(1+q^2-2q\tau)^{7/2}}, \label{eq:qdq3phi}
\end{align}
\begin{align}
\left(q \intd_q\right)^4 \phi(q)
&= \frac{\tau q-8q^2}{(1+q^2-2q\tau)^{3/2}} + \frac{12(\tau q-4q^2)(\tau q-q^2) + 9(\tau q-2q^2)^2}{(1+q^2-2q\tau)^{5/2}} \\
&\qquad +\frac{90(\tau q-q^2)^2(\tau q-2q^2)}{(1+q^2-2q\tau)^{7/2}} + \frac{105(\tau q-q^2)^4}{(1+q^2-2q\tau)^{9/2}},
\label{eq:qdq4phi}\\
\left(q\intd_q\right)^5 \phi(q)
&= \frac{\tau q-16q^2}{(1+q^2-2q\tau)^{3/2}} 
+\frac{15(\tau q-8q^2)(\tau q-q^2) + 30(\tau q-4q^2)(\tau q-2q^2)}{(1+q^2-2q\tau)^{5/2}} \\
&\qquad + \frac{150(\tau q-4q^2)(\tau q-q^2)^2 + 225(\tau q-2q^2)^2(\tau q-q^2)}{(1+q^2-2q\tau)^{7/2}}\\
&\qquad + \frac{1050(\tau q-q^2)^3(\tau q-2q^2)}{(1+q^2-2q\tau)^{9/2}}  
+ \frac{945(\tau q-q^2)^5}{(1+q^2-2q\tau)^{11/2}} .
\label{eq:qdq5phi}
\end{align}
As mentioned before, for the learning add-on, we need to compute the gradients of these terms as well:
\begin{align}
\nabla_{x} \left[\left(q\intd_q\right)^k \phi(q)\right]
\end{align} 
for $k=0,...,5$. Recall that we have
\begin{align}
q &\coloneqq h^{m}\tilde{h}^{\tilde{m}},\qquad
\tau \coloneqq \frac{x(r,\lon,t)}{|x|}\cdot \frac{\tilde{x}(r,\lon,t)}{|\tilde{x}|} \eqqcolon \xi(\lon,t) \cdot \tilde{\xi}(\lon,t).
\end{align}
Note that we need to compute the gradient with respect to $x(r,\lon,t) = h\xi(\lon,t)$. Before we consider $\nabla_x \left(q\intd_q\right)^k \phi(q)$, we first discuss some auxiliary terms.
We have
\begin{align}
\nabla q = \left\{ \begin{matrix}
\tilde{h}^{\tilde{m}}\era(\lon,t), & m=1,\\
2h\tilde{h}^{\tilde{m}}\era(\lon,t), &m=2.
\end{matrix} \right.
\end{align}
Unfortunately, this is not well-defined for $x=0$ if $m=1$ because $\era(\lon,t)$ is not declared. However, we will see that we can remedy this circumstance later. Next, we obtain\\
\begin{align}
q\nabla \tau 
&= \tilde{h}^{\tilde{m}}h^{m-1}\left( \ephi(\lon,t) \frac{1}{\sqrt{1-t^2}}\partial_\lon + \ete(\lon,t) \sqrt{1-t^2}\partial_t \right) \left(\tilde{\xi}  \cdot \xi(\lon,t)\right)\\
&= \tilde{h}^{\tilde{m}}h^{m-1}\left( \ephi(\lon,t) \frac{1}{\sqrt{1-t^2}}\left(\tilde{\xi}  \cdot \partial_\lon \xi(\lon,t)\right) + \ete(\lon,t) \sqrt{1-t^2}\left(\tilde{\xi} \cdot \partial_t \xi(\lon,t)\right) \right)\\
&= \tilde{h}^{\tilde{m}}h^{m-1}\left( \ephi(\lon,t) \frac{1}{\sqrt{1-t^2}}\left(\tilde{\xi}  \cdot \partial_\lon \era(\lon,t)\right) + \ete(\lon,t) \sqrt{1-t^2}\left(\tilde{\xi}  \cdot \partial_t \era(\lon,t)\right) \right)\\
&= \tilde{h}^{\tilde{m}}h^{m-1}\left( \ephi(\lon,t) \frac{1}{\sqrt{1-t^2}}\sqrt{1-t^2}\left(\tilde{\xi}  \cdot \ephi(\lon,t)\right) \right. \\&\qquad\qquad\qquad\qquad\left. + \ete(\lon,t) \sqrt{1-t^2}\frac{1}{\sqrt{1-t^2}}\left(\tilde{\xi} \cdot \ete(\lon,t)\right) \right)\\
&= \tilde{h}^{\tilde{m}}h^{m-1}\left( \ephi(\lon,t) \left(\tilde{\xi}  \cdot \ephi(\lon,t)\right) + \ete(\lon,t)\left(\tilde{\xi}  \cdot \ete(\lon,t)\right) \right).
\end{align}
Note that this can also be proved using the well-known $i_{\mathrm{tan}}$-tensor. Further note that, with $m\in\{1,2\}$, the term is well-defined for all $\tilde{h},\ h \in [0,1[$. \\
In $\nabla_x (q\intd_q)^k \phi(q)$, we will obtain below the latter two derivatives combined in terms of the form
\begin{align}
(\tau-2^nq)\nabla q + q\nabla \tau
\end{align}
for $n\in\nat_0$.
As mentioned before, $\nabla_x q$ has a critical point $x=0$ if $m=1$, though $\phi$ and $q\intd_q \phi$ do not. Thus, we consider this case in detail:
\begin{align}
&(\tau-2^nq)\nabla q + q\nabla \tau \\
&=\left(\tilde{\xi} \cdot \xi(\lon,t) - 2^n \tilde{h}^{\tilde{m}}h\right) \tilde{h}^{\tilde{m}}\era(\lon,t) \\
&\qquad + \tilde{h}^{\tilde{m}}\left( \ephi(\lon,t) \left(\tilde{\xi}  \cdot \ephi(\lon,t)\right) + \ete(\lon,t) \left(\tilde{\xi} \cdot \ete(\lon,t)\right) \right)\\
&=\tilde{h}^{\tilde{m}}\era(\lon,t)\left(\tilde{\xi} \cdot \era(\lon,t)\right) + \tilde{h}^{\tilde{m}}\ephi(\lon,t) \left(\tilde{\xi}  \cdot \ephi(\lon,t)\right) + \tilde{h}^{\tilde{m}}\ete(\lon,t) \left(\tilde{\xi}  \cdot \ete(\lon,t)\right)  \\
&\qquad - 2^n \tilde{h}^{\tilde{m}}h \tilde{h}^{\tilde{m}}\era(\lon,t) \\
&=\tilde{h}^{\tilde{m}-1}\tilde{x} - 2^n \tilde{h}^{2\tilde{m}}x,
\end{align}
which is well-defined for $x=0$ and $\tilde{x}=0$ for all $n\in\nat$ and can be continously extended to $x=0$. In particular, we have that 
\begin{align}
(\tau-2^nq)\nabla q + q\nabla \tau = \tilde{h}^{\tilde{m}-1}\tilde{x} = \tilde{h}^{\tilde{m}}\tilde{\xi}
\end{align}
if $x=0$. All other cases can be computed with the formulas for $\nabla q$ and $q\nabla t$ as stated before. Now, we consider $\nabla (q\intd_q)^k \phi(q)$ for $k=0,...,5$. Again the derivations are given in \cref{ssect:app1:derivderivsubterms} and we only state the results here.
\begin{align}
\nabla \phi(q) 
&= \frac{(\tau-q)\nabla q + q \nabla \tau}{(1+q^2-2q\tau)^{3/2}}, \label{eq:nablaqdq0phi}\\
\nabla \left[q\intd_q \phi(q)\right]
&= \frac{(\tau - 2q)\nabla q + q\nabla \tau}{(1+q^2-2q\tau)^{3/2}}
+\frac{3(\tau q-q^2)((\tau-q)\nabla q + q \nabla \tau)}{(1+q^2-2q\tau)^{5/2}},\label{eq:nablaqdq1phi}\\
\nabla \left[\left(q\intd_q\right)^2 \phi(q)\right]
&= \frac{(\tau -4q) \nabla q + q\nabla\tau }{(1+q^2-2q\tau)^{3/2}} \\
&\qquad + \frac{3(\tau q-2q^2)((\tau - q) \nabla q + q\nabla\tau) + 6(\tau q-q^2)((\tau -2q) \nabla q + q\nabla\tau)}{(1+q^2-2q\tau)^{5/2}}\\
&\qquad + \frac{15(\tau q-q^2)^2((\tau - q) \nabla q + q\nabla\tau)}{(1+q^2-2q\tau)^{7/2}}, \label{eq:nablaqdq2phi}\\
\end{align}
\begin{landscape}
\begin{align}
\nabla \left[\left(q\intd_q\right)^3 \phi(q)\right]
&= \frac{(\tau - 8q)\nabla q + q\nabla\tau}{(1+q^2-2q\tau)^{3/2}} 
+ \frac{3(\tau q-4q^2)((\tau - q)\nabla q + q\nabla\tau)}{(1+q^2-2q\tau)^{5/2}}\\ 
&\qquad + \frac{9((\tau - 2q)\nabla q + q\nabla\tau)(\tau q-2q^2) + 9(\tau q-q^2)((\tau - 4q)\nabla q + q\nabla\tau)}{(1+q^2-2q\tau)^{5/2}}\\ 
& \qquad + \frac{45(\tau q-q^2)(\tau q-2q^2)((\tau - q)\nabla q + q\nabla\tau) + 45(\tau q-q^2)^2((\tau - 2q)\nabla q + q\nabla\tau)}{(1+q^2-2q\tau)^{7/2}}\\
&\qquad+ \frac{105(\tau q-q^2)^3((\tau - q)\nabla q + q\nabla\tau)}{(1+q^2-2q\tau)^{9/2}}, \label{eq:nablaqdq3phi}\\ \\
\nabla \left[\left(q\intd_q\right)^4 \phi(q)\right]
&=\frac{(\tau- 16q)\nabla q + q \nabla \tau}{(1+q^2-2q\tau)^{3/2}} \\
&\qquad +\frac{3(\tau q-8q^2)((\tau- q)\nabla q + q \nabla \tau) + 12((\tau- 8q)\nabla q + q \nabla \tau)(\tau q-q^2)}{(1+q^2-2q\tau)^{5/2}}\\
&\qquad + \frac{12(\tau q-4q^2)((\tau- 2q)\nabla q + q \nabla \tau) + 18(\tau q-2q^2)((\tau- 4q)\nabla q + q \nabla \tau)}{(1+q^2-2q\tau)^{5/2}} \\
&\qquad +\frac{60(\tau q-4q^2)(\tau q-q^2) ((\tau- q)\nabla q + q \nabla \tau)
+ 45(\tau q-2q^2)^2)((\tau- q)\nabla q + q \nabla \tau)}{(1+q^2-2q\tau)^{7/2}} \\
&\qquad + \frac{180(\tau q-q^2)((\tau- 2q)\nabla q + q \nabla \tau)(\tau q-2q^2) + 90(\tau q-q^2)^2((\tau- 4q)\nabla q + q \nabla \tau)}{(1+q^2-2q\tau)^{7/2}} \\
&\qquad +\frac{630(\tau q-q^2)^2(\tau q-2q^2)((\tau- q)\nabla q + q \nabla \tau) + 420(\tau q-q^2)^3((\tau- 2q)\nabla q + q \nabla \tau)}{(1+q^2-2q\tau)^{9/2}}\\
&\qquad +\frac{945(\tau q-q^2)^4((\tau- q)\nabla q + q \nabla \tau)}{(1+q^2-2q\tau)^{11/2}},\label{eq:nablaqdq4phi}
\end{align}
\begin{align}
\nabla \left[\left(q\intd_q\right)^5 \phi(q)\right]
&= \frac{(\tau -32q)\nabla q + q\nabla \tau}{(1+q^2-2q\tau)^{3/2}} \\
&\qquad + \frac{3(\tau q-16q^2)((\tau - q)\nabla q + q \nabla \tau)+ 15((\tau - 16q)\nabla q + q \nabla \tau)(\tau q-q^2)}{(1+q^2-2q\tau)^{5/2}} \\
&\qquad + \frac{15(\tau q-8q^2)((\tau - 2q)\nabla q + q \nabla \tau) + 30((\tau - 8q)\nabla q + q \nabla \tau)(\tau q-2q^2)}{(1+q^2-2q\tau)^{5/2}} \\
&\qquad + \frac{30(\tau q-4q^2)((\tau - 4q)\nabla q + q \nabla \tau)}{(1+q^2-2q\tau)^{5/2}} \\
&\qquad + \frac{75(\tau q-8q^2)(\tau q-q^2)((\tau - q)\nabla q + q \nabla \tau) + 150(\tau q-4q^2)(\tau q-2q^2)((\tau - q)\nabla q + q \nabla \tau)}{(1+q^2-2q\tau)^{7/2}}\\
&\qquad + \frac{150((\tau - 8q)\nabla q + q \nabla \tau)(\tau q-q^2)^2 + 300(\tau q-4q^2)(\tau q-q^2)((\tau - 2q)\nabla q + q \nabla \tau)}{(1+q^2-2q\tau)^{7/2}}\\
&\qquad + \frac{450(\tau q-2q^2)((\tau - 4q)\nabla q + q \nabla \tau)(\tau q-q^2) + 225(\tau q-2q^2)^2((\tau - 2q)\nabla q + q \nabla \tau)}{(1+q^2-2q\tau)^{7/2}}\\
&\qquad + \frac{1050(\tau q-4q^2)(\tau q-q^2)^2((\tau - q)\nabla q + q \nabla \tau) + 1575(\tau q-2q^2)^2(\tau q-q^2)((\tau - q)\nabla q + q \nabla \tau)}{(1+q^2-2q\tau)^{9/2}}\\
&\qquad + \frac{3150(\tau q-q^2)^2((\tau - 2q)\nabla q + q \nabla \tau)(\tau q-2q^2) + 1050(\tau q-q^2)^3((\tau - 4q)\nabla q + q \nabla \tau)}{(1+q^2-2q\tau)^{9/2}} \\
&\qquad +\frac{9450(\tau q-q^2)^3(\tau q-2q^2)((\tau - q)\nabla q + q \nabla \tau) + 4725(\tau q-q^2)^4((\tau - 2q)\nabla q + q \nabla \tau)}{(1+q^2-2q\tau)^{11/2}} \\
&\qquad + \frac{10395(\tau q-q^2)^5((\tau - q)\nabla q + q \nabla \tau)}{(1+q^2-2q\tau)^{13/2}}. \label{eq:nablaqdq5phi}
\end{align}
\end{landscape}

\section{Numerics}
\label{sect:numerics}

\subsection{General setting}
\label{ssect:numerics:genset}

As mentioned before, we approximate the Earth's surface with the unit sphere. The efficient singular value decomposition of the operator $\T$ is then given in \cref{eq:gravpotatorbit}. With this, the upward continued value of the gravitational potential can be computed via Level 2 (or comparable) data. That means, we use downloadable data that gives us the coefficients of a truncation of \cref{eq:gravpotatorbit}. In detail, we use the EGM2008, see e.\,g. \cite{Pavlisetal2012}, as well as GRACE-FO Level-2 Monthly Geopotential Spherical Har\-mon\-ics XXX Release 6.0 (RL06), where XXX stands for JPL, CSR and GFZ, from December 2022 (data files GSM-2\_2022335-2022365\_GRFO\_XXXXX\_BB01\_0600, where XXXXX stands for GFZOP, JPLEM and UTCSR). The latter can be downloaded from NASA’s Earth Observing System Data and Information System (EOSDIS) (\textcolor{blue}{\url{https://www.earthdata.nasa.gov/}}). For the GRACE data, we build the ensemble mean among the different origins as was advised in \cite{Sakumuraetal2014}. 
Removing the Earth's ellipticity is conventionally done by subtracting the normal field. However, we consider it sufficient for our purposes to only neglect the degree 2 part of the field. Therefore, we obtain 
\begin{align}
V\left(\sigma\eta\right) 
= \left( \T f \right) \left(\sigma\eta\right) 
\approx \sum_{n=3}^N \sum_{j=-n}^n f_{n,j}\sigma^{-n-1}Y_{n,j}\left(\eta\right),
\label{eq:gravpotatorbit_trunc}
\end{align}
where the truncation stops for degree $N=2190$ and up to order $2159$ in the case of the EGM2008 and for degree $N=96$ and all orders for GRACE-FO. We perform two experiments: we evaluate this truncation for the EGM2008 and GRACE-FO data, respectively, on the kinematic satellite orbits of GRACE-FO-1 and -2 from December 2022. The orbits originate from the Institute of Geodesy at the TU Graz, see \cite{Suesser-Rechberger2022}, and can be obtained from \textcolor{blue}{\url{https://www.tugraz.at/institute/ifg/downloads/satellite-orbit-products}}. For each twin satellite GRACE-FO-1 and -2, there exists a set of $\approx 8\,500$ Cartesian coordinates for each mission day (data files GRACEFO-x\_kinematicOrbit\_2022-12-dd, where x is either 1 or 2 and dd stands for 01,...,31). That makes roughly 250\,000 data points for each month for each twin satellite. Using both data at once gives us the desired 500\,000 grid points. In particular, in December 2022, we have 529\,832 grid points. Note that the coordinates of the grid points are given in the celestial reference frame. For conversion to the terrestial frame, we use software from the International Astronomical Union's SOFA service (see \textcolor{blue}{\url{https://www.iausofa.org/index.html}}). In particular, we considered section 5 in the documentation given in \cite{SOFAdocu}. We consider the tracks as satellite orbits above a perfectly spherical Earth. For this purpose, we adjust the radial coordinates of the orbit tracks for the case of the Earth as a unit ball. The corresponding satellite heights vary around 500\,km, see \cref{fig:satheights}. 
\begin{wrapfigure}{r}{0.5\textwidth}
\centering
\begin{subfigure}{0.5\textwidth}
	\includegraphics[width=\textwidth]{"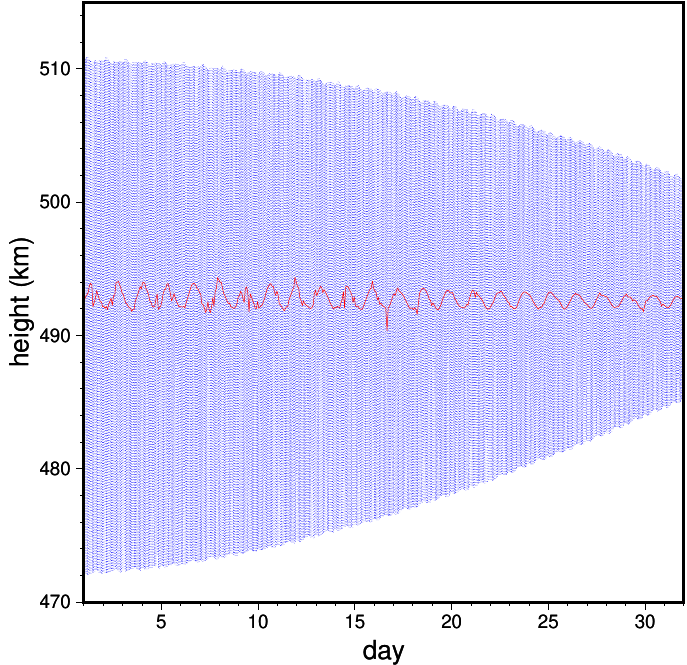"}
	\caption{Heights over the spherical Earth of the kinematic orbit tracks of GRACE-FO 1 in December 2022. Every fifth data point is plotted. The red line gives the mean per revolution. Note that GRACE-FO 2 repeats this pattern.}
\end{subfigure}\\[\baselineskip]
\begin{subfigure}{0.5\textwidth}
	\includegraphics[width=\textwidth]{"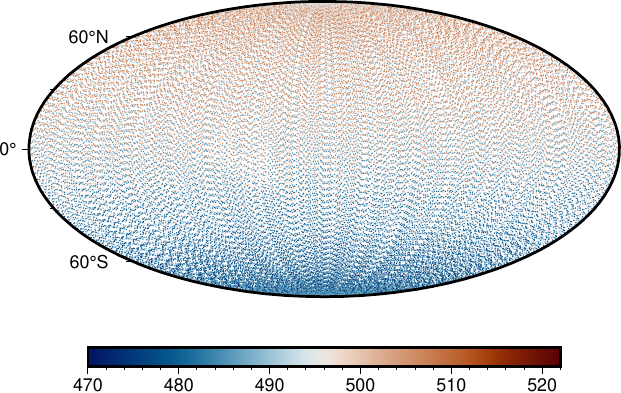"}
	\caption{Distribution of satellite tracks. Every fifth data point is plotted. The colour gives the absolute satellite height in km.}
\end{subfigure}
\caption{Distribution of kinematic orbits.}
\label{fig:satheights}
\end{wrapfigure}
Further, we perturb the data with noise where we do a simplified stochastic approach here. The perturbed data $y^\delta$ given by 
\begin{align}
y^\delta_i = y_i \cdot \left( 1 + 0.05 \cdot \varepsilon_i\right)
\label{eq:noise}
\end{align}
for the unperturbed data $y_i = \left( \T f \right) \left(\sigma^i\eta^i\right)$ and a unit normally distributed random number $\varepsilon_i$. Our noise model is certainly not realistic due to the very complicated anisotropic error behaviour of, for example, GRACE and GRACE-FO. Moreover, additive noise might better represent the practical situation. However, it suffices for our purposes of testing the general numerical performance of our method in the case of a large amount of data. Since we continuously try to improve our method, we will also take into account more realistic noise models in the future. \\
For evaluation of the solution and the approximation, we consider a Driscoll Healy grid, see e.g. \cite{DriscollHealy1994}. We use 901 latitudinal steps and 1801 longitudinal steps which makes up a grid of $\kappa = 1662701$ points in total which we use to plot the approximation. Note that this grid is larger then the used data.

The resolution of the used data and the data grids are evaluated in \cref{tab:resolution}. Note that the resolution of the data depends on the number of extrema associated to the largest spherical harmonic order as we have at most $2|j|$ extrema at the equator and the resolution of the data grid depends on the (approximate) number of grid points at the equator. For the orbit tracks, we note that each satellite circles the Earth roughly every 90 minutes. In December 2022, this amounts to approximately 496 revolutions for each satellite. Per revolution, the two satellites cross the equator twice, hence, we have $2\cdot 2\cdot 496 = 1984$ points at the equator, which is a resolution of $\approx$ 20 km at best. We see that the EGM2008 is a bit undersampled with the used orbit tracks as well as the Driscoll-Healy grid. The GRACE-FO data is, however, more accurately sampled than necessary. Hence, we assume that these data are suitable to describe the behaviour of the LRFMP with these many grid points. Note that the Driscoll-Healy grid obviously does not have a better resolution at the equator than the satellite orbits. However, recall that the grid points of the Driscoll-Healy grid have a higher resolution towards the poles.
\begin{table}[t]
\centering
\begin{tabular}{lcc}
& resolution at 500 km & resolution at surface\\
\hline
EGM2008 	& 9.998 	& 9.271 \\
GRACE FO & 224.953 	& 208.490 \\
\hline
kinematic satellite orbits & 20.176 & ---\\
Driscoll Healy grid & ---& 22.227\\
\end{tabular}
\caption{Comparison of the resolution of the data and the used data grids. All values in km. Note that the Driscoll Healy grid is used only for evaluation at Earth's modelled surface and the kinematic orbits are only used in space.}
\label{tab:resolution}
\end{table}
The Driscoll-Healy grid is also used for the relative root mean square error (RRMSE)
\begin{align}
\sqrt{\frac{\sum_{i=1}^{\kappa} (f_\nu(\tilde{\eta}^i) - f(\tilde{\eta}^i))^2}{\sum_{i=1}^{\kappa} (f(\tilde{\eta}^i))^2}},
\end{align}
where $f$ is the evaluation of \cref{eq:gravpotatorbit_trunc} with $\sigma=1$ and $f_\nu$ is the approximation after $\nu$ iterations. Besides the relative RMSE, we also consider the relative data error (RDE) $\|R^\nu\|_{\real^\ell}/\|R^0\|_{\real^\ell}$ and the absolute approximation error. 

\begin{wraptable}{r}{0.5\textwidth}
\centering
\begin{tabular}{lcc}
$\lambda$	& RDE & RRMSE\\
\hline
$10^{-8}$ 	& 0.050\,496 	& 0.176\,465 \\
$10^{-9}$ 	& 0.050\,169 	& 0.176\,140 \\
$10^{-10}$ 	& 0.050\,055 	& 0.180\,222 \\
\hline
$10^{-8}$ 	& 0.050\,589 	& 0.175\,218 \\
$10^{-9}$ 	& 0.050\,239 	& 0.175\,039 \\
$10^{-10}$ 	& 0.050\,149 	& 0.181\,298 \\
\hline
\end{tabular}
\caption{Comparison of the RDE and the RRMSE for different regularization parameters $\lambda$ for the EGM2008 data (upper row) and the GRACE FO data (lower row).}
\label{tab:regpars}
\end{wraptable} 

We terminate the algorithm after 1\,600 iterations for the EGM2008 data and after 1\,400 iterations for the GRACE FO data. Note that the former truncates the potential later. Thus, we allow more iterations.

We tested several regularization parameters and, for the most suitable ones, obtained the relative data and root mean square error as given \cref{tab:regpars}. We consider these parameters the most suitable ones as the relative data error has almost reached the noise level of $5\%$. Obviously, we have lower data errors for a lower regularization parameter. However, it is well-known that an approximation might be affected by overfitting if the regularization parameter is chosen too low. As we aim here for a demonstration regarding the applicability to a large amount of data, we have chosen to present the results for the parameter $\lambda = 10^{-8}$ as they appear to have overall the least overfitting identifyable in their visualization. We refer an in-depth regularization study to future research. 

For the optimization problems in the learning add-on, we utilize methods from the NLOpt library, see \cite{NLOpt2019}. In particular, we use the ORIG\_DIRECT\_L (globally) and the SLSQP (locally) algorithms. The constraints are tightened by $10^{-8}$. The absolute tolerance for the change of the objective function between two successive iterates and of the iterates themselves is $10^{-8}$. We allow 10\,000 function evaluations and 600 seconds for each optimization.
As the starting dictionary, we define 
\begin{align}
\left[N^\mathrm{s}\right]_\SH &= \left\{ Y_{n,j}\ \left|\ n=0,...,96;\ j=-n,...,n \right.\right\}	\\
\left[B_K^\mathrm{s}\right]_\APK &= \left\{ K(x,\cdot)\ \left|\ |x| = 0.94,\ \frac{x}{|x|} \in X^\mathrm{s} \right. \right\}\\
\left[B_W^\mathrm{s}\right]_\APW &= \left\{ W(x,\cdot)\ \left|\ |x| = 0.94,\ \frac{x}{|x|} \in X^\mathrm{s} \right.\right\}\\
\dic^{\mathrm{s}} &= \left[N^\mathrm{s}\right]_{\SH}  \cup  \left[B_K^\mathrm{s}\right]_{\APK} \cup \left[B_W^\mathrm{s}\right]_{\APW}
\end{align}
with a regularly distributed Reuter grid $X^\mathrm{s}$ of $123$ grid points. Thus, the starting dictionary contains 9655 trial functions. This allows the experiments of the LIPMP algorithms to run on an HPC node of 256 GB RAM with 64 CPUs. Note that the number of CPUs is more relevant than the RAM in our experiments.

The corresponding code is available at \textcolor{blue}{\protect{\url{https://doi.org/10.5281/zenodo.8223771}}} under the licence CC BY-NC-SA 3.0 Germany. Note that we also include the RFMP, the ROFMP and the LROFMP in this source code. Their implementation is tested for errors and obviously is the most efficient variant of these methods which has been published so far. For more details on them, we refer the interested reader to the previously mentioned literature. Further note that, in particular, the (L)ROFMP needs additional parameters to be set before running the method.

\begin{figure}
\centering
\includegraphics[width=.95\textwidth]{"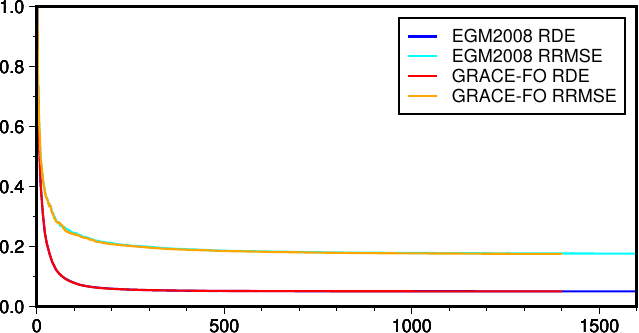"}
\caption{RDEs and RRMSEs along the iterations.}
\label{fig:errors}
\end{figure}

\begin{landscape}
\begin{figure}
\centering
\includegraphics[width=.45\textwidth]{"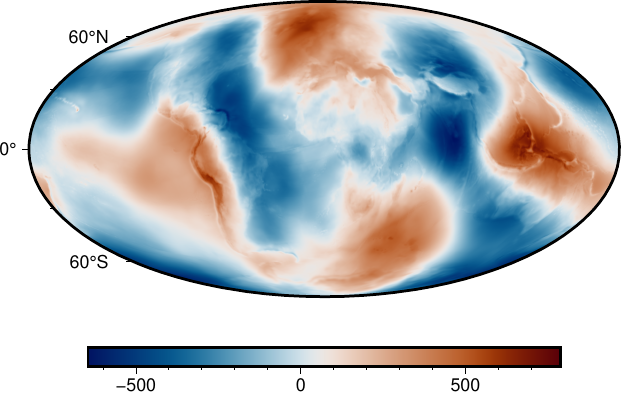"}
\includegraphics[width=.45\textwidth]{"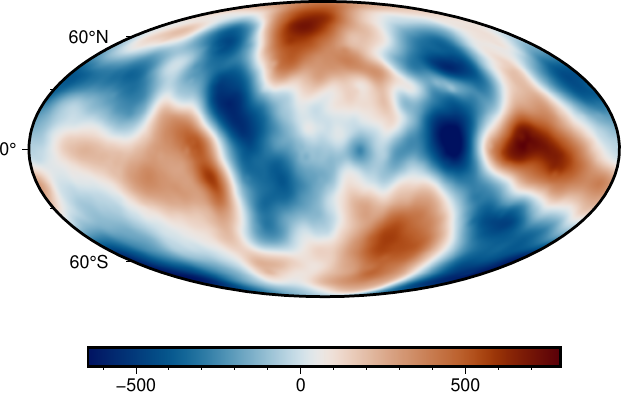"}
\includegraphics[width=.45\textwidth]{"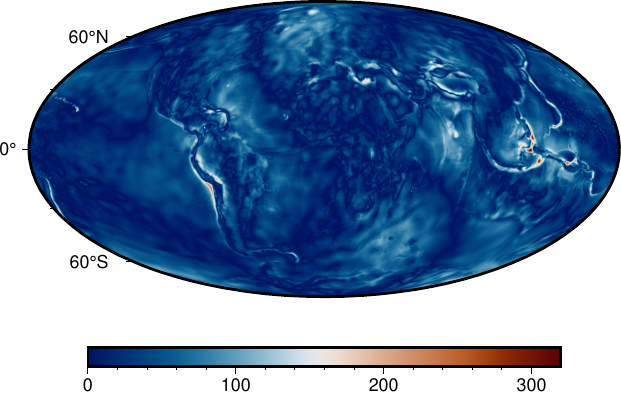"}
\includegraphics[width=.45\textwidth]{"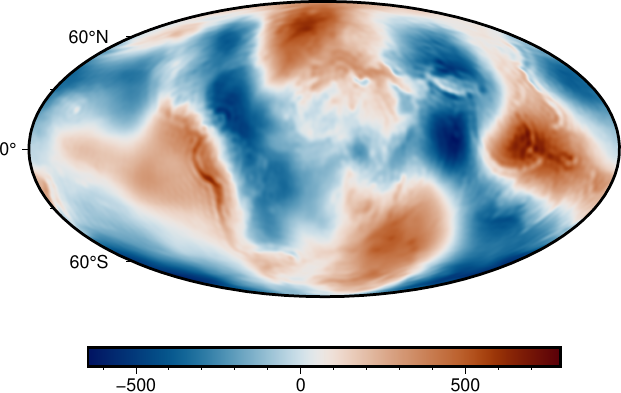"}
\includegraphics[width=.45\textwidth]{"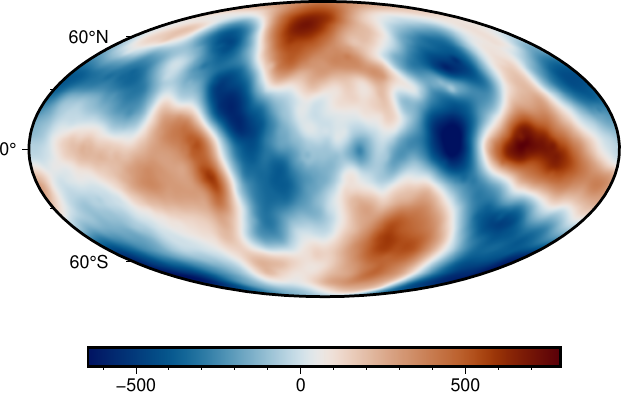"}
\includegraphics[width=.45\textwidth]{"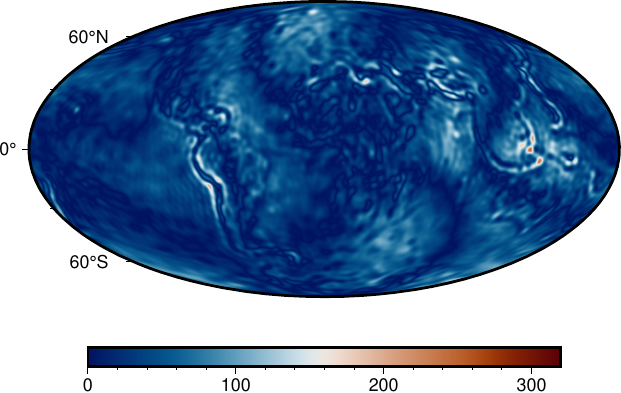"}
\caption{Left column: the respective truncated gravitational potential. Middle column: the obtained approximation. Right column: the absolute approximation error. The upper row contains these visualizations regarding the EGM2008 data. The lower row contains the respective ones regarding the GRACE FO data. All values in $\mathrm{m}^2/\mathrm{s}^2$. The scales are adapted for better comparability.}
\label{fig:pics}
\end{figure}
\end{landscape}

\subsection{Results and evaluation}
\label{ssect:numerics:res_eval}

In \cref{tab:regpars}, we see the final RDE and RRMSE for the chosen regularization parameter $\lambda=10^{-8}$. The RDE is almost at the noise level for both data sets. Moreover, also the RRMSE is similar at approximately $17.5\%$. This shows the severe ill-posedness of the considered problem. In \cref{fig:errors}, we see how these errors decrease along the iterations. Both errors plummet in the beginning. In later iterations, only small progress is made. This gives rise to the question whether we could allow also much less iterations for similar results. This emphasizes that the LRFMP is able to reproduce high-dimensional data in a sparse best basis.

In \cref{fig:pics}, we see the truncated gravitational potentials (left column) of the EGM2008 data (upper row) and the used GRACE FO data (lower row), respectively. Moreover, we give the respective approximations (middle column) and the absolute approximation errors (right column). Note that the scales are adapted for better comparability.  First of all, we notice that, for both data sets, the remaining errors lie in regions with high local structures like the Andes, the Himalayas and the Pacific ring of fire. This is due to the used satellite height of roughly 500 km which makes for a loss of data that cannot be retrieved in the inversion process (compare with \cite{Schneideretal2022}). Further, it can be explained by the remaining $\approx 5\%$ RDE which may include some data alongside pure noise as well. When comparing the absolute approximation error of the EGM2008 data and the GRACE FO data, we notice that the extreme error values are larger in the case of the EGM2008 data. Note, however, that the EGM2008 data gives the gravitational potential truncated at degree 2190 while the GRACE-FO data truncates the series at degree 96. Moreover, besides these extreme values, most of the errors appear to be of the same size.

All in all, the results show that we were able to model the LRFMP for high-dimensional data. We discover the expectable behaviour of approximation methods in the results. In particular, we see that remaining errors lie in regions with high local structures. Moreover, we also gained insights into the next aspects that should be discussed for further development of our experiments. These are the determination of a minimal number of iterations for similar good results and the investigation of parameter choice methods for the regularization parameter. Further, as we mentioned before, the use of a realistic noise model should also be investigated next.

\section{Conclusion and Outlook}
\label{sect:cons}
We started our considerations with the modelling of the gravitational potential from satellite data. This is necessary in several geoscientific applications such as the visualization of climate change phenomena. 

In our approach, we aim to approximate the potential with one of the LIPMP algorithms. This enables us to combine different types of trial functions: (global) polynomials and (local) radial basis functions. Thus, the approach unites the specific benefits of the various basis functions. In contrast to previous methods, it also harmonizes the ideas of the spherical harmonic representation and the use of local trial functions, like e.\,g.\, mascons. In \cite{Schneideretal2022,Schneider2020}, we extensively showed that the methods can in principle be used for gravitational modelling. However, the grids of data points were rather small ($\approx$ 12\,500) and artificial. Here, we showed what can be done to improve the efficiency in order to be able to use a sensible amount of data, such as 500\,000 grid points, on a realistic satellite track distribution. 

Firstly, we reconsidered some modelling aspects like the choice of types of trial functions as well as the most suitable variant of the LIPMPs from the previous experiment settings. We decided to abstain from optimizing Slepian funcions (that we used before) and choose the LRFMP for upscaling our experiments. Furthermore, the revision of the code also pointed towards certain series that hindered the use of a large amount of data. These series represent inner products of two radial basis functions and/or wavelets in use. Here, we showed that we can substitute the inefficient and inaccurate truncated series with a fast and accurate closed form. 

In our experiments, we used the revised LRFMP algorithm for approximating the gravitational potential from the EGM2008 as well as the GRACE-FO data from December 2022. We computed the gravitational data on kinematic orbit tracks. The results show that, in this setting, we obtain good approximative solutions for this exponentially ill-posed inverse problem with the LRFMP while remaining errors lie in regions with highly local structures. The latter can be explained amongst others by the use of roughly 500\,km satellite height. 

Hence, we showed that the methods can be tailored for high-dimensional tasks even though this modelling is quite task specific. Therefore, it is an open question whether the discussed aspects can be transferred to other inverse problems easily. 

In the development of our methods towards competitiveness in the geodetic community, next, we have to consider the missing aspects that are necessary for implementing real GRACE and GRACE-FO data. This includes most importantly testing observable-depending data like Level 1B or range rates, but also allowing the true and stochastically complicated noise of GRACE and GRACE-FO. 

Both aspects are starting points for future research though we consider the former to be implemented more easily. In particular, if the data source changes, so does the representation of the data. That means, mathematically, if the operator $\T$ changes, so do depending terms such as $\T d$ for a dictionary element $d$. With the status-quo reached in this paper, in particular, the ability to work with a sensible amount of data, it appears that this challenge can now be tackled.

On the remaining aspect of a true stochastic GRACE (-FO) noise model, we have a more conservative outlook as of now. This is because it is highly difficult to appraise the behaviour of the methods for this experiment design. In general, the appropriate modelling of noise characteristics within the theory of inverse problems has been more stepmotherly treated so far but is currently also a topic of ongoing research in the inverse problems community. The further enhancement of the LIPMPs might profit from it.

\newpage

\section*{Declarations}

\paragraph{Funding} VM gratefully acknowledges the support by the German Research Foundation (DFG; Deutsche Forschungsgemeinschaft), project MI 655/14-1. 
\paragraph{Conflicts of interest/Competing interests} Not applicable.
\paragraph{Availability of data, material and code}  The Level 2 GRACE data are downloaded from NASA’s Earth Observing System Data and Information System (EOSDIS) \textcolor{blue}{\url{https://www.earthdata.nasa.gov/}}. The satellite tracks used are downloaded from \textcolor{blue}{\url{https://www.tugraz.at/institute/ifg/downloads/satellite-orbit-products}}. For details, see also \cite{Suesser-Rechberger2022}. The conversion of the position coordinates of the satellite tracks from the celestial to the terrestial frame is based on \textcolor{blue}{\url{https://www.iausofa.org/index.html}}. From these data, we obtained the upward continued potential as well as the evaluated potential at the Earth's surface. These point values and the source code are available at \textcolor{blue}{\protect{\url{https://doi.org/10.5281/zenodo.8223771}}} under the licence CC BY-NC-SA 3.0 Germany.
\paragraph{Author's contributions} The research was carried out during the postdoc project of N. Schneider. In this partially DFG-funded project, VM is the principal investigator. N. Sneeuw is the project partner from geosciences. Both supervise the ongoing project and assist N. Schneider. 

\footnotesize
\bibliography{biblioneu}

\begin{thebibliography}{60}
\providecommand{\natexlab}[1]{#1}
\providecommand{\url}[1]{\texttt{#1}}
\expandafter\ifx\csname urlstyle\endcsname\relax
  \providecommand{\doi}[1]{doi: #1}\else
  \providecommand{\doi}{doi: \begingroup \urlstyle{rm}\Url}\fi

\bibitem[{Albertella} et~al.(1999){Albertella}, {Sans\`o}, and
  {Sneeuw}]{Albertellaetal1999}
A.~{Albertella}, F.~{Sans\`o}, and N.~{Sneeuw}.
\newblock {Band-limited functions on a bounded spherical domain: the Slepian
  problem on the sphere}.
\newblock \emph{Journal of Geodesy}, 73\penalty0 (9):\penalty0 436--447, 1999.
\newblock \protect{\url{https://doi.org/10.1007/PL00003999}}.

\bibitem[{Baur}(2014)]{Baur2014}
O.~{Baur}.
\newblock {Gravity field of planetary bodies}.
\newblock In E.~{Grafarend}, editor, \emph{{Encyclopedia of Geodesy}}, pages
  1--6. Springer International Publishing Switzerland AG, Cham, 2014.
\newblock \protect{\url{https://doi.org/10.1007/978-3-319-02370-0_46-1}}.

\bibitem[{Chen} et~al.(2022){Chen}, {Cazenave}, {Dahle}, {Llovel}, {Panet},
  {Pfeffer}, and {Moreira}]{Chenetal2022}
J.~{Chen}, A.~{Cazenave}, C.~{Dahle}, W.~{Llovel}, I.~{Panet}, J.~{Pfeffer},
  and L.~{Moreira}.
\newblock {Applications and challenges of GRACE and GRACE Follow-On satellite
  gravimetry}.
\newblock \emph{Surveys in Geophysics}, 43\penalty0 (1):\penalty0 305--345,
  2022.
\newblock \protect{\url{https://doi.org/10.1007/s10712-021-09685-x}}.

\bibitem[{Chen} et~al.(2021){Chen}, {Shen}, {Kusche}, {Chen}, {Chen}, and
  {Zhang}]{Chenetal2021}
Q.~{Chen}, Y.~{Shen}, J.~{Kusche}, W.~{Chen}, T.~{Chen}, and X.~{Zhang}.
\newblock {High-Resolution GRACE Monthly Spherical Harmonic Solutions.}
\newblock \emph{Journal of Geophysical Research Solid Earth}, 126\penalty0
  (1):\penalty0 e2019JB018892, 2021.
\newblock \protect{\url{https://doi.org/10.1029/2019JB018892}}.

\bibitem[{Deuflhard}(1976)]{Deuflhard1976}
P.~{Deuflhard}.
\newblock {On algorithms for the summation of certain special functions}.
\newblock \emph{Computing}, 17\penalty0 (1):\penalty0 37--48, 1976.
\newblock \protect{\url{https://doi.org/10.1007/BF02252258}}.

\bibitem[{Devaraju} and {Sneeuw}(2017)]{Devaraju2017}
B.~{Devaraju} and N.~{Sneeuw}.
\newblock {The polar form of the spherical harmonic spectrum: implications for
  filtering GRACE data}.
\newblock \emph{Journal of Geodesy}, 91\penalty0 (12):\penalty0 1475--1489,
  2017.
\newblock \protect{\url{https://doi.org/10.1007/s00190-017-1037-7}}.

\bibitem[{Driscoll} and {Healy}(1994)]{DriscollHealy1994}
J.~R. {Driscoll} and D.~M. {Healy}.
\newblock {Computing Fourier transforms and convolutions on the 2-sphere}.
\newblock \emph{Advances in Applied Mathematics}, 15\penalty0 (2):\penalty0
  202--250, 1994.
\newblock \protect{\url{https://doi.org/10.1006/aama.1994.1008}}.

\bibitem[{Engl} et~al.(1996){Engl}, {Hanke}, and {Neubauer}]{Engletal1996}
H.~W. {Engl}, M.~{Hanke}, and A.~{Neubauer}.
\newblock \emph{{Regularization of Inverse Problems}}.
\newblock {Mathematics and Its Applications}. Kluwer Academic Publishers,
  Dordrecht, 1996.

\bibitem[{Fenlason} and {Stallman}(1998)]{GNUgprof}
J.~{Fenlason} and R.~{Stallman}.
\newblock {GNU gprof}, 1998.
\newblock
  \protect{\url{https://ftp.gnu.org/old-gnu/Manuals/gprof-2.9.1/html_mono/gprof.html}},
  last accessed 22 December 2022.

\bibitem[{Fischer} and {Michel}(2013{\natexlab{a}})]{Fischeretal2013-1}
D.~{Fischer} and V.~{Michel}.
\newblock {Automatic best-basis selection for geophysical tomographic inverse
  problems}.
\newblock \emph{Geophysical Journal International}, 193\penalty0 (3):\penalty0
  1291--1299, 2013{\natexlab{a}}.
\newblock \protect{\url{https://doi.org/10.1093/gji/ggt038}}.

\bibitem[{Fischer} and {Michel}(2013{\natexlab{b}})]{Fischeretal2013-2}
D.~{Fischer} and V.~{Michel}.
\newblock {Inverting GRACE gravity data for local climate effects}.
\newblock \emph{Journal of Geodetic Science}, 3\penalty0 (3):\penalty0
  151--162, 2013{\natexlab{b}}.
\newblock \protect{\url{https://doi.org/10.2478/jogs-2013-0019}}.

\bibitem[{Flechtner} et~al.(2014{\natexlab{a}}){Flechtner}, {Morton},
  {Watkins}, and {Webb}]{Flechtneretal2014}
F.~{Flechtner}, P.~{Morton}, M.~{Watkins}, and F.~{Webb}.
\newblock {Status of the GRACE follow-on mission}.
\newblock In U.~{Marti}, editor, \emph{Gravity, Geoid and Height Systems.
  International Association of Geodesy Symposia}, volume 141, pages 117--121.
  Springer, Cham, 2014{\natexlab{a}}.

\bibitem[{Flechtner} et~al.(2014{\natexlab{b}}){Flechtner}, {Sneeuw}, and
  {Schuh}]{Flechtneretal2014-2}
F.~{Flechtner}, N.~{Sneeuw}, and W.-D. {Schuh}, editors.
\newblock \emph{{Observation of the System Earth from Space -- CHAMP, GRACE,
  GOCE and future missions}}.
\newblock Springer, Berlin, 2014{\natexlab{b}}.

\bibitem[{Freeden} and {Michel}(1999)]{Freedenetal1999}
W.~{Freeden} and V.~{Michel}.
\newblock {Constructive approximation and numerical methods in geodetic
  research today -- an attempt at a categorization based on an uncertainty
  principle}.
\newblock \emph{Journal of Geodesy}, 73\penalty0 (9):\penalty0 452--465, 1999.
\newblock \protect{\url{https://doi.org/10.1007/PL00004001}}.

\bibitem[{Freeden} and {Michel}(2004)]{Freedenetal2004}
W.~{Freeden} and V.~{Michel}.
\newblock \emph{{Multiscale Potential Theory with Applications to Geoscience}}.
\newblock Birkh\"auser, Boston, 2004.

\bibitem[{Freeden} and {Schreiner}(1998)]{Freedenetal1998-1}
W.~{Freeden} and M.~{Schreiner}.
\newblock {Orthogonal and non-orthogonal multiresolution analysis, scale
  discrete and exact fully discrete wavelet transform on the sphere}.
\newblock \emph{Constructive Approximation}, 14\penalty0 (4):\penalty0
  493--515, 1998.
\newblock \protect{\url{https://doi.org/10.1007/s003659900087}}.

\bibitem[{Freeden} and {Schreiner}(2009)]{Freedenetal2009}
W.~{Freeden} and M.~{Schreiner}.
\newblock \emph{{Spherical Functions of Mathematical Geosciences -- A Scalar,
  Vectorial, and Tensorial Setup}}.
\newblock Springer, Berlin, 2009.

\bibitem[{Freeden} and {Windheuser}(1996)]{Freedenetal1996}
W.~{Freeden} and U.~{Windheuser}.
\newblock {Spherical wavelet transform and its discretization}.
\newblock \emph{Advances in Computational Mathematics}, 5\penalty0
  (1):\penalty0 51--94, 1996.
\newblock \protect{\url{https://doi.org/10.1007/BF02124735}}.

\bibitem[{Freeden} et~al.(1998){Freeden}, {Gervens}, and
  {Schreiner}]{Freedenetal1998}
W.~{Freeden}, T.~{Gervens}, and M.~{Schreiner}.
\newblock \emph{{Constructive Approximation on the Sphere -- with Applications
  to Geomathematics}}.
\newblock Oxford University Press, Oxford, 1998.

\bibitem[{Freeden} et~al.(2018){Freeden}, {Michel}, and
  {Simons}]{Freedenetal2018}
W.~{Freeden}, V.~{Michel}, and F.~J. {Simons}.
\newblock {Spherical harmonics based special function systems and constructive
  approximation methods}.
\newblock In W.~{Freeden} and M.~Z. {Nashed}, editors, \emph{Handbook of
  Mathematical Geodesy. Geosystems Mathematics.}, pages 753--819. Birkhäuser,
  Cham, 2018.
\newblock \protect{\url{https://doi.org/10.1007/978-3-319-57181-2_12}}.

\bibitem[{Güntner} et~al.(2023){Güntner}, {Gerdener}, {Boergens}, {Kusche},
  {Kollet}, {Dobslaw}, {Hartick}, {Sharifi}, and {Flechtner}]{Guentneretal2023}
A.~{Güntner}, H.~{Gerdener}, E.~{Boergens}, J.~{Kusche}, S.~{Kollet},
  H.~{Dobslaw}, C.~{Hartick}, E.~{Sharifi}, and F.~{Flechtner}.
\newblock {Veränderungen der Wasserspeicherung in Deutschland seit 2002 aus
  Beobachtungen der Satellitengravimetrie -- Water storage changes in Germany
  since 2002 from satellite gravimetry observations}.
\newblock \emph{Hydrologie \& Wasserbewirtschaftung}, 67\penalty0 (2):\penalty0
  74--89, 2023.
\newblock \protect{\url{https://doi.org/10.5675/HyWa_2023.2_1}}.

\bibitem[{Hager} and {Wellein}(2011)]{Hageretal2011}
G.~{Hager} and G.~{Wellein}.
\newblock \emph{{Introduction to High Performance Computing for Scientists and
  Engineers}}.
\newblock Chapman \& Hall/CRC, 2011.

\bibitem[{International Astronomical Union, Standards Of Fundamental
  Astronomy}(2021)]{SOFAdocu}
{International Astronomical Union, Standards Of Fundamental Astronomy}.
\newblock {SOFA Tools for Earth Attitude}, 2021.
\newblock
  \protect{\url{https://www.iausofa.org/2021_0512_C/sofa/sofa_pn_c.pdf}}, last
  accessed 21 April 2023.

\bibitem[IPCC(2023)]{IPCC2023}
IPCC.
\newblock \emph{{Climate Change 2023: Synthesis Report. Contribution of Working
  Groups I, II and III to the Sixth Assessment Report of the Intergovernmental
  Panel on Climate Change}}.
\newblock IPCC, Geneva, Switzerland, 2023.
\newblock \protect{\url{https://doi.org/10.59327/IPCC/AR6-9789291691647}}.

\bibitem[{Johnson}(2019)]{NLOpt2019}
S.~G. {Johnson}.
\newblock \emph{{The NLopt nonlinear-optimization package}}, 2019.
\newblock \protect{\url{http://github.com/stevengj/nlopt}} and
  \protect{\url{https://nlopt.readthedocs.io/en/latest/}}, both last accessed
  29 June 2023.

\bibitem[{Kontak} and {Michel}(2019)]{Kontaketal2018-1}
M.~{Kontak} and V.~{Michel}.
\newblock {The Regularized Weak Functional Matching Pursuit for linear inverse
  problems}.
\newblock \emph{Journal of Inverse and Ill-Posed Problems}, 27\penalty0
  (3):\penalty0 317--340, 2019.
\newblock \protect{\url{https://doi.org/10.1515/jiip-2018-0013}}.

\bibitem[{Kusche} et~al.(2012){Kusche}, {Klemann}, and {Bosch}]{Kuscheetal2012}
J.~{Kusche}, V.~{Klemann}, and W.~{Bosch}.
\newblock {Mass distribution and mass transport in the Earth system}.
\newblock \emph{Journal of Geodynamics}, 59--60:\penalty0 1--8, 2012.
\newblock \protect{\url{https://doi.org/10.1016/j.jog.2012.03.003}}.

\bibitem[{Landerer} et~al.(2020){Landerer}, {Flechtner}, {Save}, {Webb},
  {Bandikova}, {Bertiger}, {Bettadpur}, {Byun}, {Dahle}, {Dobslaw},
  {Fahnestock}, {Harvey}, {Kang}, {Kruizinga}, {Loomis}, {McCullough},
  {Murböck}, {Nagel}, {Paik}, {Pie}, {Poole}, {Strekalov}, {Tamisiea}, {Wang},
  {Watkins}, {Wen}, {Wiese}, and {Yuan}]{Landereretal2020}
F.~W. {Landerer}, F.~M. {Flechtner}, H.~{Save}, F.~H. {Webb}, T.~{Bandikova},
  W.~I. {Bertiger}, S.~V. {Bettadpur}, S.~H. {Byun}, C.~{Dahle}, H.~{Dobslaw},
  E.~{Fahnestock}, N.~{Harvey}, Z.~{Kang}, G.~L.~H. {Kruizinga}, B.~D.
  {Loomis}, C.~{McCullough}, M.~{Murböck}, P.~{Nagel}, M.~{Paik}, N.~{Pie},
  S.~{Poole}, D.~{Strekalov}, M.~E. {Tamisiea}, F.~{Wang}, M.~M. {Watkins},
  H.-Y. {Wen}, D.~N. {Wiese}, and D.-N. {Yuan}.
\newblock {Extending the global mass change data record: GRACE Follow-On
  instrument and science data performance}.
\newblock \emph{Geophysical Research Letters}, 47\penalty0 (12):\penalty0
  e2020GL088306, 2020.
\newblock \protect{\url{https://doi.org/10.1029/2020GL088306}}.

\bibitem[{Louis}(1989)]{Louis1989}
A.~K. {Louis}.
\newblock \emph{{Inverse und schlecht gestellte Probleme}}.
\newblock Teubner, Stuttgart, 1989.

\bibitem[{Luthcke} et~al.(2013){Luthcke}, {Sabaka}, {Loomis}, {Arendt},
  {McCarthy}, and {Camp}]{Luthckeetal2013}
S.~B. {Luthcke}, T.~J. {Sabaka}, B.~D. {Loomis}, A.~A. {Arendt}, J.~J.
  {McCarthy}, and J.~{Camp}.
\newblock {Antarctica, Greenland and Gulf of Alaska land-ice evolution from an
  iterated GRACE global mascon solution.}
\newblock \emph{Journal of Glaciology}, 59\penalty0 (216):\penalty0 613--631,
  2013.
\newblock \protect{\url{https://doi.org/10.3189/2013JoG12J147}}.

\bibitem[{Michel}(2005)]{Michel2005}
V.~{Michel}.
\newblock {Regularized wavelet-based multiresolution recovery of the harmonic
  mass density distribution from data of the Earth's gravitational field at
  satellite height}.
\newblock \emph{Inverse Problems}, 21\penalty0 (3):\penalty0 997--1025, 2005.
\newblock \protect{\url{https://dx.doi.org/10.1088/0266-5611/21/3/013}}.

\bibitem[{Michel}(2013)]{Michel2013}
V.~{Michel}.
\newblock \emph{{Lectures on Constructive Approximation -- Fourier, Spline, and
  Wavelet Methods on the Real Line, the Sphere, and the Ball}}.
\newblock Birkhäuser, New York, 2013.

\bibitem[{Michel}(2015)]{Michel2015-2}
V.~{Michel}.
\newblock {RFMP -- An iterative best basis algorithm for inverse problems in
  the geosciences}.
\newblock In W.~{Freeden}, M.~Z. {Nashed}, and T.~{Sonar}, editors,
  \emph{Handbook of Geomathematics}, pages 2121--2147. Springer, Berlin,
  Heidelberg, {2nd} edition, 2015.
\newblock \protect{\url{https://doi.org/10.1007/978-3-642-27793-1_93-1}}.

\bibitem[{Michel}(2022)]{Michel2020}
V.~{Michel}.
\newblock \emph{{Geomathematics -- Modelling and Solving Mathematical Problems
  in Geodesy and Geophysics}}.
\newblock Cambridge University Press, Cambridge, 2022.

\bibitem[{Michel} and {Orzlowski}(2017)]{Micheletal2017-1}
V.~{Michel} and S.~{Orzlowski}.
\newblock {On the convergence theorem for the Regularized Functional Matching
  Pursuit (RFMP) algorithm}.
\newblock \emph{GEM -- International Journal on Geomathematics}, 8\penalty0
  (2):\penalty0 183--190, 2017.
\newblock \protect{\url{https://doi.org/10.1007/s13137-017-0095-6}}.

\bibitem[{Michel} and {Schneider}(2020)]{Micheletal2018-1}
V.~{Michel} and N.~{Schneider}.
\newblock {A first approach to learning a best basis for gravitational field
  modelling}.
\newblock \emph{GEM - International Journal on Geomathematics}, 11\penalty0
  (1):\penalty0 Article 9, 2020.
\newblock \protect{\url{https://doi.org/10.1007/s13137-020-0143-5}}.

\bibitem[{Michel} and {Telschow}(2016)]{Micheletal2016-1}
V.~{Michel} and R.~{Telschow}.
\newblock {The Regularized Orthogonal Functional Matching Pursuit for ill-posed
  inverse problems}.
\newblock \emph{SIAM Journal on Numerical Analysis}, 54\penalty0 (1):\penalty0
  262--287, 2016.
\newblock \protect{\url{https://doi.org/10.1137/141000695}}.

\bibitem[{Moritz}(2010)]{Moritz2010}
H.~{Moritz}.
\newblock {Classical physical geodesy}.
\newblock In W.~{Freeden}, M.~Z. {Nashed}, and T.~{Sonar}, editors,
  \emph{{Handbook of Geomathematics}}, pages 253--289. Springer, Berlin,
  Heidelberg, 2nd edition, 2010.
\newblock \protect{\url{https://doi.org/10.1007/978-3-642-01546-5_6}}.

\bibitem[{Morozov}(1966)]{Morozov1966}
V.~A. {Morozov}.
\newblock {On the solution of functional equations by the method of
  regularization}.
\newblock \emph{Doklady Akademii Nauk SSSR}, 167\penalty0 (3):\penalty0
  510--512, 1966.
\newblock {Translation in Sov. Math., Dokl. 7, 414-417 (1966)}.

\bibitem[{M\"uller}(1966)]{Mueller1966}
C.~{M\"uller}.
\newblock \emph{{Spherical Harmonics}}.
\newblock Springer, Berlin, 1966.

\bibitem[NASA(2020)]{NasaConsensu}
NASA.
\newblock {Global Climate Change: Scientific Consensus}, 2020.
\newblock \protect\url{https://climate.nasa.gov/scientific-consensus/}, last
  accessed 29 June 2023.

\bibitem[{NASA Jet Propulsion Laboratory}(2020)]{GRACEdata}
{NASA Jet Propulsion Laboratory}.
\newblock {GRACE Tellus}, 2020.
\newblock \protect\url{https://grace.jpl.nasa.gov/}, last accessed 29 June
  2023.

\bibitem[{Pavlis} et~al.(2012){Pavlis}, {Holmes}, {Kenyon}, and
  {Factor}]{Pavlisetal2012}
N.~K. {Pavlis}, S.~A. {Holmes}, S.~C. {Kenyon}, and J.~K. {Factor}.
\newblock {The development and evaluation of the Earth Gravitational Model 2008
  (EGM2008)}.
\newblock \emph{Journal of Geophysical Research: Solid Earth}, 117\penalty0
  (B4), 2012.
\newblock \protect{\url{https://doi.org/10.1029/2011JB008916}}. Correction in
  Volume 118, Issue 5, \protect{\url{https://doi.org/10.1002/jgrb.50167}}.

\bibitem[{Prakash} et~al.(2020){Prakash}, {Radhamani}, and
  {Priyalatha}]{Prakashetal2020}
E.~{Prakash}, V.~{Radhamani}, and S.~P.~R. {Priyalatha}.
\newblock {Determination of Earth's mass density distribution based on
  satellite data}.
\newblock \emph{Advances in Mathematics: Scientific Journal}, 9\penalty0
  (9):\penalty0 7223--7233, 2020.
\newblock \protect{\url{https://doi.org/10.37418/amsj.9.9.71}}.

\bibitem[{Rieder}(2003)]{Rieder2003}
A.~{Rieder}.
\newblock \emph{{Keine Probleme mit inversen Problemen. Eine Einführung in
  ihre stabile Lösung}}.
\newblock Vieweg, Wiesbaden, 2003.

\bibitem[{Saemian} et~al.(2022){Saemian}, {Tourian}, {AghaKouchak}, {Madani},
  and {Sneeuw}]{Saemianetal2022}
P.~{Saemian}, M.~J. {Tourian}, A.~{AghaKouchak}, K.~{Madani}, and N.~{Sneeuw}.
\newblock {How much water did Iran lose over the last two decades?}
\newblock \emph{Journal of Hydrology: Regional Studies}, 41:\penalty0 101095,
  2022.
\newblock \protect{\url{https://doi.org/10.1016/j.ejrh.2022.101095}}.

\bibitem[{Sakumura} et~al.(2014){Sakumura}, {Bettadpur}, and
  {Bruinsma}]{Sakumuraetal2014}
C.~{Sakumura}, S.~{Bettadpur}, and S.~{Bruinsma}.
\newblock {Ensemble prediction and intercomparison analysis of GRACE
  time-variable gravity field models}.
\newblock \emph{Geophysical Research Letters}, 41\penalty0 (5):\penalty0
  1389--1397, 2014.
\newblock \protect{\url{https://doi.org/10.1002/2013GL058632}}.

\bibitem[{Save} et~al.(2016){Save}, {Bettadpur}, and {Tapley}]{Saveetal2016}
H.~{Save}, S.~{Bettadpur}, and B.~D. {Tapley}.
\newblock {High-resolution CSR GRACE RL05 mascons.}
\newblock \emph{Journal of Geophysical Research Solid Earth}, 121\penalty0
  (10):\penalty0 7547--7569, 2016.
\newblock \protect{\url{https://doi.org/10.1002/2016JB013007}}.

\bibitem[{Schmidt} et~al.(2008){Schmidt}, {Flechtner}, {Meyer}, Neumayer,
  {Dahle}, {König}, and {Kusche}]{Schmidtetal2008}
R.~{Schmidt}, F.~{Flechtner}, U.~{Meyer}, K.~H. Neumayer, C.~{Dahle},
  R.~{König}, and J.~{Kusche}.
\newblock {Hydrological signals observed by the GRACE satellites}.
\newblock \emph{Surveys in Geophysics}, 29\penalty0 (4--5):\penalty0 319--334,
  2008.
\newblock \protect{\url{https://doi.org/10.1007/s10712-008-9033-3}}.

\bibitem[{Schneider}(2020)]{Schneider2020}
N.~{Schneider}.
\newblock \emph{{Learning Dictionaries for Inverse Problems on the Sphere}}.
\newblock PhD thesis, University of Siegen, Geomathematics Group, 2020.
\newblock \protect{\url{http://dx.doi.org/10.25819/ubsi/5431}}.

\bibitem[{Schneider} and {Michel}(2022)]{Schneideretal2022}
N.~{Schneider} and V.~{Michel}.
\newblock {A dictionary learning add-on for spherical downward continuation}.
\newblock \emph{Journal of Geodesy}, 96\penalty0 (4):\penalty0 Article 21,
  2022.
\newblock \protect{\url{https://doi.org/10.1007/s00190-022-01598-w}}.

\bibitem[{Simons} et~al.(2006){Simons}, {Dahlen}, and
  {Wieczorek}]{Simonsetal2006}
F.~J. {Simons}, F.~A. {Dahlen}, and M.~A. {Wieczorek}.
\newblock {Spatiospectral concentration on a sphere}.
\newblock \emph{SIAM Review}, 48\penalty0 (3):\penalty0 504--536, 2006.
\newblock \protect{\url{https://doi.org/10.1137/S0036144504445765}}.

\bibitem[{Suesser-Rechberger} et~al.(2022){Suesser-Rechberger}, {Krauss},
  {Strasser}, and {Mayer-Guerr}]{Suesser-Rechberger2022}
B.~{Suesser-Rechberger}, S.~{Krauss}, S.~{Strasser}, and T.~{Mayer-Guerr}.
\newblock {Improved precise kinematic LEO orbits based on the raw observation
  approach}.
\newblock \emph{Advances in Space Research}, 69\penalty0 (10):\penalty0
  3559--3570, 2022.
\newblock \protect{\url{https://doi.org/10.1016/j.asr.2022.03.014}}.

\bibitem[{Tapley} et~al.(2004){Tapley}, {Bettadpur}, {Watkins}, and
  {Reigber}]{Tapleyetal2004}
B.~D. {Tapley}, S.~{Bettadpur}, M.~{Watkins}, and C.~{Reigber}.
\newblock {The gravity recovery and climate experiment: mission overview and
  early results}.
\newblock \emph{Geophysical Research Letters}, 31\penalty0 (9):\penalty0
  L09607, 2004.
\newblock \protect{\url{https://doi.org/10.1029/2004GL019920}}.

\bibitem[{Tapley} et~al.(2019){Tapley}, {Watkins}, {Flechtner}, {Reigber},
  {Bettadpur}, {Rodell}, {Sasgen}, {Famiglietti}, {Landerer}, {Chambers},
  {Reager}, {Gardner}, {Save}, {Ivins}, {Swenson}, {Boening}, {Dahle}, {Wiese},
  {Dobslaw}, {Tamisiea}, and {Velicogna}]{Tapleyetal2019}
B.~D. {Tapley}, M.~M. {Watkins}, F.~{Flechtner}, C.~{Reigber}, S.~{Bettadpur},
  M.~{Rodell}, I.~{Sasgen}, J.~S. {Famiglietti}, F.~W. {Landerer}, D.~P.
  {Chambers}, J.~T. {Reager}, A.~S. {Gardner}, H.~{Save}, E.~R. {Ivins}, S.~C.
  {Swenson}, C.~{Boening}, C.~{Dahle}, D.~N. {Wiese}, H.~{Dobslaw}, M.~E.
  {Tamisiea}, and I.~{Velicogna}.
\newblock {Contributions of GRACE to understanding climate change}.
\newblock \emph{Nature Climate Change}, 9\penalty0 (5):\penalty0 358--369,
  2019.
\newblock \protect{\url{https://doi.org/10.1038/s41558-019-0456-2}}.

\bibitem[{Telschow} et~al.(2018){Telschow}, {Gerhards}, and
  {Rother}]{Telschowetal2018}
R.~{Telschow}, C.~{Gerhards}, and M.~{Rother}.
\newblock {On the approximation of spatial structures of global tidal magnetic
  field models}.
\newblock \emph{Annales Geophysicae}, 36\penalty0 (5):\penalty0 1393--1402,
  2018.
\newblock \protect{\url{https://doi.org/10.5194/angeo-36-1393-2018}}.

\bibitem[{The University of Texas at Austin, Centre for Space
  Research}(2020)]{GRACEdata2}
{The University of Texas at Austin, Centre for Space Research}.
\newblock {GRACE} gravity recovery and climate experiment, 2020.
\newblock \protect\url{http://www2.csr.utexas.edu/grace/}, last accessed 3
  March 2020.

\bibitem[{Watkins} et~al.(2015){Watkins}, {Wiese}, {Yuan}, {Boening}, and
  {Landerer}]{Watkinsetal2015}
M.~M. {Watkins}, D.~N. {Wiese}, D.-N. {Yuan}, C.~{Boening}, and F.~W.
  {Landerer}.
\newblock {Improved methods for observing Earth's time variable mass
  distribution with GRACE using spherical cap mascons.}
\newblock \emph{Journal of Geophysical Research Solid Earth}, 120\penalty0
  (4):\penalty0 2648--2671, 2015.
\newblock \protect{\url{https://doi.org/10.1002/2014JB011547}}.

\bibitem[{Wiese} et~al.(2022){Wiese}, {Bienstock}, {Blackwood}, {Chrone},
  {Loomis}, {Sauber}, {Rodell}, {Baize}, {Bearden}, {Case}, {Horner},
  {Luthcke}, {Reager}, {Srinivasan}, {Tsaoussi}, {Webb}, {Whitehurst}, and
  {Zlotnicki}]{Wieseetal2022}
D.~N. {Wiese}, B.~{Bienstock}, C.~{Blackwood}, J.~{Chrone}, B.~D. {Loomis},
  J.~{Sauber}, M.~{Rodell}, R.~{Baize}, D.~{Bearden}, K.~{Case}, S.~{Horner},
  S.~{Luthcke}, J.~T. {Reager}, M.~{Srinivasan}, L.~{Tsaoussi}, F.~{Webb},
  A.~{Whitehurst}, and V.~{Zlotnicki}.
\newblock {The mass change designated observable study: overview and results}.
\newblock \emph{Earth and Space Science}, 9\penalty0 (8):\penalty0
  e2022EA002311, 2022.
\newblock \protect{\url{https://doi.org/10.1029/2022EA002311}}.

\bibitem[{Windheuser}(1995)]{Windheuser1995}
U.~{Windheuser}.
\newblock \emph{{Sphärische Wavelets: Theorie und Anwendung in der
  Physikalischen Geodäsie}}.
\newblock PhD thesis, University of Kaiserslautern, Geomathematics Group, 1995.

\end{thebibliography}
\normalfont

\appendix
\newpage
\section{Derivations in detail}
\label{sect:app1}
\subsection{Derivation of subterms for closed formula of inner products}
\label{ssect:app1:derivsubterms}
Note that we also give the names of the terms used in the source code in each final formula after the derivation.
We obtain
\begin{align}
\phi(q) 
&= \frac{1}{(1+q^2-2q\tau)^{1/2}} 
\label{eq:qdq0phi}
\eqqcolon \frac{1}{\mathrm{denom}}\\ \notag\\
q \intd_q \phi(q) 
&= -\frac{1}{2} \frac{1}{(1+q^2-2q\tau)^{3/2}}(2q-2\tau)q
= \frac{\tau q-q^2}{(1+q^2-2q\tau)^{3/2}}
\label{eq:qdq1phi}
\eqqcolon \frac{\mathrm{tqmqsq}}{\mathrm{denom3}}\\ \notag\\
\left(q \intd_q\right)^2 \phi(q)
&= q\left( \frac{\tau-2q}{(1+q^2-2q\tau)^{3/2}} - \frac{3}{2}\frac{\tau q-q^2}{(1+q^2-2q\tau)^{5/2}}(2q-2\tau) \right)\\
&= \frac{\tau q-2q^2}{(1+q^2-2q\tau)^{3/2}} + \frac{3(\tau q-q^2)^2}{(1+q^2-2q\tau)^{5/2}}
\label{eq:qdq2phi}\\\notag\\
&\eqqcolon \frac{\mathrm{tqm2qsq}}{\mathrm{denom3}} + \frac{3\mathrm{tqmqsq2}}{\mathrm{denom5}}\\
\left(q \intd_q\right)^3 \phi(q)
&= q\left( \frac{\tau-4q}{(1+q^2-2q\tau)^{3/2}} - \frac{3}{2}\frac{\tau q-2q^2}{(1+q^2-2q\tau)^{5/2}}(2q-2\tau) \right. \\ 
&\qquad \left. + \frac{6(\tau q-q^2)(\tau-2q)}{(1+q^2-2q\tau)^{5/2}} - \frac{5}{2} \frac{3(\tau q-q^2)^2}{(1+q^2-2q\tau)^{7/2}}(2q-2\tau)\right)\\
&= q\left( \frac{\tau-4q}{(1+q^2-2q\tau)^{3/2}} + \frac{3(\tau q-2q^2)(\tau-q)}{(1+q^2-2q\tau)^{5/2}}\right. \\ 
&\qquad \left. + \frac{6(\tau q-q^2)(\tau-2q)}{(1+q^2-2q\tau)^{5/2}} + \frac{15(\tau q-q^2)^2(\tau-q)}{(1+q^2-2q\tau)^{7/2}}\right)\\
&= \frac{\tau q-4q^2}{(1+q^2-2q\tau)^{3/2}} + \frac{3(\tau q-2q^2)(\tau q-q^2)}{(1+q^2-2q\tau)^{5/2}}\\ 
&\qquad + \frac{6(\tau q-q^2)(\tau q-2q^2)}{(1+q^2-2q\tau)^{5/2}} + \frac{15(\tau q-q^2)^2(\tau q-q^2)}{(1+q^2-2q\tau)^{7/2}}\\
&= \frac{\tau q-4q^2}{(1+q^2-2q\tau)^{3/2}} + \frac{9(\tau q-q^2)(\tau q-2q^2)}{(1+q^2-2q\tau)^{5/2}} + \frac{15(\tau q-q^2)^3}{(1+q^2-2q\tau)^{7/2}}
\label{eq:qdq3phi}\\ \notag \\
&\eqqcolon \frac{\mathrm{tqm4qsq}}{\mathrm{denom3}} + \frac{9\mathrm{tqmqsq*tqm2qsq}}{\mathrm{denom5}} + \frac{15\mathrm{tqmqsq3}}{\mathrm{denom7}}\\
\left(q \intd_q\right)^4 \phi(q)
&=   q\left(
\frac{\tau-8q}{(1+q^2-2q\tau)^{3/2}} 
- \frac{3}{2}\frac{\tau q-4q^2}{(1+q^2-2q\tau)^{5/2}}(2q-2\tau)
\right.\\
&\qquad  + \frac{9(\tau-2q)(\tau q-2q^2)}{(1+q^2-2q\tau)^{5/2}} 
+ \frac{9(\tau q-q^2)(\tau-4q)}{(1+q^2-2q\tau)^{5/2}} 
- \frac{5}{2}\frac{9(\tau q-q^2)(\tau q-2q^2)}{(1+q^2-2q\tau)^{7/2}}(2q-2\tau)
\\ 
&\qquad \left.
+ \frac{45(\tau q-q^2)^2(\tau-2q)}{(1+q^2-2q\tau)^{7/2}}
- \frac{7}{2}\frac{15(\tau q-q^2)^3}{(1+q^2-2q\tau)^{9/2}}(2q-2\tau)
\right)
\end{align}
\begin{align}
&=   q\left(
\frac{\tau-8q}{(1+q^2-2q\tau)^{3/2}} 
+ \frac{3(\tau q-4q^2)(\tau-q)}{(1+q^2-2q\tau)^{5/2}}
\right.\\
&\qquad  + \frac{9(\tau-2q)(\tau q-2q^2)}{(1+q^2-2q\tau)^{5/2}} 
+ \frac{9(\tau q-q^2)(\tau-4q)}{(1+q^2-2q\tau)^{5/2}} 
+\frac{45(\tau q-q^2)(\tau q-2q^2)(\tau-q)}{(1+q^2-2q\tau)^{7/2}}
\\ 
&\qquad \left.
+ \frac{45(\tau q-q^2)^2(\tau-2q)}{(1+q^2-2q\tau)^{7/2}}
+ \frac{105(\tau q-q^2)^3(\tau-q)}{(1+q^2-2q\tau)^{9/2}}
\right)\\
&=   
\frac{\tau q-8q^2}{(1+q^2-2q\tau)^{3/2}} 
+ \frac{3(\tau q-4q^2)(\tau q-q^2)}{(1+q^2-2q\tau)^{5/2}}
\\
&\qquad  + \frac{9(\tau q-2q^2)(\tau q-2q^2)}{(1+q^2-2q\tau)^{5/2}} 
+ \frac{9(\tau q-q^2)(\tau q-4q^2)}{(1+q^2-2q\tau)^{5/2}} 
+\frac{45(\tau q-q^2)(\tau q-2q^2)(\tau q-q^2)}{(1+q^2-2q\tau)^{7/2}}
\\ 
&\qquad 
+ \frac{45(\tau q-q^2)^2(\tau q-2q^2)}{(1+q^2-2q\tau)^{7/2}}
+ \frac{105(\tau q-q^2)^3(\tau q-q^2)}{(1+q^2-2q\tau)^{9/2}}\\
&=   
\frac{\tau q-8q^2}{(1+q^2-2q\tau)^{3/2}} 
+ \frac{3(\tau q-4q^2)(\tau q-q^2)}{(1+q^2-2q\tau)^{5/2}}
\\
&\qquad + \frac{9(\tau q-2q^2)^2}{(1+q^2-2q\tau)^{5/2}} 
+ \frac{9(\tau q-q^2)(\tau q-4q^2)}{(1+q^2-2q\tau)^{5/2}} 
+\frac{45(\tau q-q^2)^2(\tau q-2q^2)}{(1+q^2-2q\tau)^{7/2}}
\\ 
&\qquad 
+ \frac{45(\tau q-q^2)^2(\tau q-2q^2)}{(1+q^2-2q\tau)^{7/2}}
+ \frac{105(\tau q-q^2)^4}{(1+q^2-2q\tau)^{9/2}}\\
&=   
\frac{\tau q-8q^2}{(1+q^2-2q\tau)^{3/2}} + \frac{12(\tau q-4q^2)(\tau q-q^2) + 9(\tau q-2q^2)^2}{(1+q^2-2q\tau)^{5/2}} \\
&\qquad +\frac{90(\tau q-q^2)^2(\tau q-2q^2)}{(1+q^2-2q\tau)^{7/2}} + \frac{105(\tau q-q^2)^4}{(1+q^2-2q\tau)^{9/2}}
\label{eq:qdq4phi}\\ \notag\\
&\eqqcolon \frac{\mathrm{tqm8qsq}}{\mathrm{denom3}} + \frac{12\mathrm{tqm4qsq*tqmqsq} + 9\mathrm{tqm2qsq2}}{\mathrm{denom5}} \\
&\qquad +\frac{90\mathrm{tqmqsq2*tqm2qsq}}{\mathrm{denom7}} + \frac{105\mathrm{tqmqsq4}}{\mathrm{denom9}}\\ \notag\\
&\left(q\intd_q\right)^5 \phi(q)\\
&=q\left( 
\frac{\tau-16q}{(1+q^2-2q\tau)^{3/2}} 
-\frac{3}{2}\frac{\tau q-8q^2}{(1+q^2-2q\tau)^{5/2}}(2q-2\tau)\right.\\
&\qquad + \frac{12(\tau-8q)(\tau q-q^2) + 12(\tau q-4q^2)(\tau-2q) + 18(\tau q-2q^2)(\tau-4q)}{(1+q^2-2q\tau)^{5/2}} \\
&\qquad \qquad - \frac{5}{2}\frac{12(\tau q-4q^2)(\tau q-q^2) + 9(\tau q-2q^2)^2}{(1+q^2-2q\tau)^{7/2}} (2q-2\tau)\\
&\qquad +\frac{180(\tau q-q^2)(\tau-2q)(\tau q-2q^2) + 90(\tau q-q^2)^2(\tau-4q)}{(1+q^2-2q\tau)^{7/2}} \\
& \qquad \qquad - \frac{7}{2}\frac{90(\tau q-q^2)^2(\tau q-2q^2)}{(1+q^2-2q\tau)^{9/2}}(2q-2\tau)\\
&\qquad + \left. \frac{420(\tau q-q^2)^3(\tau-2q)}{(1+q^2-2q\tau)^{9/2}} 
-\frac{9}{2} \frac{105(\tau q-q^2)^4}{(1+q^2-2q\tau)^{11/2}}(2q-2\tau) \right)
\end{align}
\begin{align} 
&=q\left( 
\frac{\tau-16q}{(1+q^2-2q\tau)^{3/2}} 
+\frac{3(\tau q-8q^2)(\tau-q)}{(1+q^2-2q\tau)^{5/2}}\right.\\
&\qquad + \frac{12(\tau-8q)(\tau q-q^2) + 12(\tau q-4q^2)(\tau-2q) + 18(\tau q-2q^2)(\tau-4q)}{(1+q^2-2q\tau)^{5/2}} \\
&\qquad \qquad + \frac{60(\tau q-4q^2)(\tau q-q^2)(\tau-q) + 45(\tau q-2q^2)^2(\tau-q)}{(1+q^2-2q\tau)^{7/2}}\\
&\qquad +\frac{180(\tau q-q^2)(\tau-2q)(\tau q-2q^2) + 90(\tau q-q^2)^2(\tau-4q)}{(1+q^2-2q\tau)^{7/2}} \\
& \qquad \qquad + \frac{630(\tau q-q^2)^2(\tau q-2q^2)(\tau-q)}{(1+q^2-2q\tau)^{9/2}}\\
&\qquad + \left. \frac{420(\tau q-q^2)^3(\tau-2q)}{(1+q^2-2q\tau)^{9/2}} 
+ \frac{945(\tau q-q^2)^4(\tau-q)}{(1+q^2-2q\tau)^{11/2}} \right)\\
&= \frac{\tau q-16q^2}{(1+q^2-2q\tau)^{3/2}} 
+\frac{3(\tau q-8q^2)(\tau q-q^2)}{(1+q^2-2q\tau)^{5/2}}\\
&\qquad + \frac{12(\tau q-8q^2)(\tau q-q^2) + 12(\tau q-4q^2)(\tau q-2q^2) + 18(\tau q-2q^2)(\tau q-4q^2)}{(1+q^2-2q\tau)^{5/2}} \\
&\qquad \qquad + \frac{60(\tau q-4q^2)(\tau q-q^2)(\tau q-q^2) + 45(\tau q-2q^2)^2(\tau q-q^2)}{(1+q^2-2q\tau)^{7/2}}\\
&\qquad +\frac{180(\tau q-q^2)(\tau q-2q^2)(\tau q-2q^2) + 90(\tau q-q^2)^2(\tau q-4q^2)}{(1+q^2-2q\tau)^{7/2}} \\
& \qquad \qquad + \frac{630(\tau q-q^2)^2(\tau q-2q^2)(\tau q-q^2)}{(1+q^2-2q\tau)^{9/2}}\\
&\qquad + \frac{420(\tau q-q^2)^3(\tau q-2q^2)}{(1+q^2-2q\tau)^{9/2}} 
+ \frac{945(\tau q-q^2)^4(\tau q-q^2)}{(1+q^2-2q\tau)^{11/2}} \\
&= \frac{\tau q-16q^2}{(1+q^2-2q\tau)^{3/2}} 
+\frac{3(\tau q-8q^2)(\tau q-q^2)}{(1+q^2-2q\tau)^{5/2}}\\
&\qquad + \frac{12(\tau q-8q^2)(\tau q-q^2) + 30(\tau q-4q^2)(\tau q-2q^2)}{(1+q^2-2q\tau)^{5/2}} \\
&\qquad \qquad + \frac{60(\tau q-4q^2)(\tau q-q^2)^2 + 45(\tau q-2q^2)^2(\tau q-q^2)}{(1+q^2-2q\tau)^{7/2}}\\
&\qquad +\frac{180(\tau q-q^2)(\tau q-2q^2)^2 + 90(\tau q-q^2)^2(\tau q-4q^2)}{(1+q^2-2q\tau)^{7/2}} \\
& \qquad \qquad + \frac{630(\tau q-q^2)^3(\tau q-2q^2)}{(1+q^2-2q\tau)^{9/2}}\\
&\qquad + \frac{420(\tau q-q^2)^3(\tau q-2q^2)}{(1+q^2-2q\tau)^{9/2}} 
+ \frac{945(\tau q-q^2)^5}{(1+q^2-2q\tau)^{11/2}} \\
&= \frac{\tau q-16q^2}{(1+q^2-2q\tau)^{3/2}} \\
&\qquad +\frac{15(\tau q-8q^2)(\tau q-q^2) + 30(\tau q-4q^2)(\tau q-2q^2)}{(1+q^2-2q\tau)^{5/2}} \\
&\qquad + \frac{150(\tau q-4q^2)(\tau q-q^2)^2 + 225(\tau q-2q^2)^2(\tau q-q^2)}{(1+q^2-2q\tau)^{7/2}}\\
&\qquad + \frac{1050(\tau q-q^2)^3(\tau q-2q^2)}{(1+q^2-2q\tau)^{9/2}} \\
&\qquad + \frac{945(\tau q-q^2)^5}{(1+q^2-2q\tau)^{11/2}} 
\label{eq:qdq5phi}
\end{align}
\begin{align}
&\eqqcolon \frac{\mathrm{tqm16qsq}}{\mathrm{denom3}} \\
&\qquad +\frac{15\mathrm{tqm8qsq*tqmqsq} + 30\mathrm{tqm4qsq*tqm2qsq}}{\mathrm{denom5}} \\
&\qquad + \frac{150\mathrm{tqm4qsq*tqmqsq2} + 225\mathrm{tqm2qsq2*tqmqsq}}{\mathrm{denom7}}\\
&\qquad + \frac{1050\mathrm{tqmqsq3*tqm2qsq}}{\mathrm{denom9}} \\
&\qquad + \frac{945\mathrm{tqmqsq5}}{\mathrm{denom11}}
\end{align}

\subsection{Derivation of derivatives of subterms for closed formula of inner products}
\label{ssect:app1:derivderivsubterms}

\begin{align}
&\nabla \phi(q) \\
&= \nabla\left(\frac{1}{(1+q^2-2q\tau)^{1/2}}\right)
= -\frac{1}{2}\frac{1}{(1+q^2-2q\tau)^{3/2}} \nabla(1+q^2-2q\tau)\\
&= -\frac{1}{2}\frac{2q\nabla q - 2\tau\nabla q - 2q \nabla \tau}{(1+q^2-2q\tau)^{3/2}} 
= \frac{\tau\nabla q - q\nabla q + q \nabla \tau}{(1+q^2-2q\tau)^{3/2}} \\
&= \frac{(\tau-q)\nabla q + q \nabla \tau}{(1+q^2-2q\tau)^{3/2}} 
\label{eq:nablaqdq0phi}\\\notag\\
&\eqqcolon \frac{\mathrm{ID}}{\mathrm{denom3}}\\ \notag \\
&\nabla \left[q\intd_q \phi(q)\right]\\
&= \nabla\left(\frac{\tau q-q^2}{(1+q^2-2q\tau)^{3/2}}\right)\\
&= \frac{q\nabla \tau + \tau \nabla q - 2q\nabla q}{(1+q^2-2q\tau)^{3/2}}
- \frac{3}{2}\frac{(\tau q-q^2)(2q\nabla q - 2\tau\nabla q - 2q \nabla \tau)}{(1+q^2-2q\tau)^{5/2}}\\
&= \frac{(\tau - 2q)\nabla q + q\nabla \tau}{(1+q^2-2q\tau)^{3/2}}
+\frac{3(\tau q-q^2)((\tau-q)\nabla q + q \nabla \tau)}{(1+q^2-2q\tau)^{5/2}}
\label{eq:nablaqdq1phi}\\\notag\\
&\eqqcolon \frac{\mathrm{ID2}}{\mathrm{denom3}}
+\frac{3\mathrm{tqmqsq*ID}}{\mathrm{denom5}}\\\notag \\
&\nabla \left[\left(q\intd_q\right)^2 \phi(q)\right]\\
&= \nabla\left(\frac{\tau q-2q^2}{(1+q^2-2q\tau)^{3/2}} + \frac{3(\tau q-q^2)^2}{(1+q^2-2q\tau)^{5/2}}\right)\\
&= \frac{q\nabla\tau + \tau \nabla q - 4q\nabla q}{(1+q^2-2q\tau)^{3/2}} 
- \frac{3}{2}\frac{(\tau q-2q^2)(2q\nabla q - 2\tau\nabla q - 2q \nabla \tau)}{(1+q^2-2q\tau)^{5/2}}\\
&\qquad + \frac{6(\tau q-q^2)(q \nabla \tau + \tau \nabla q - 2q\nabla q)}{(1+q^2-2q\tau)^{5/2}} - \frac{5}{2}\frac{3(\tau q-q^2)^2(2q\nabla q - 2\tau\nabla q - 2q \nabla \tau)}{(1+q^2-2q\tau)^{7/2}}
\end{align}
\begin{align}
&= \frac{(\tau -4q) \nabla q + q\nabla\tau }{(1+q^2-2q\tau)^{3/2}} 
+ \frac{3(\tau q-2q^2)((\tau - q) \nabla q + q\nabla\tau)}{(1+q^2-2q\tau)^{5/2}}\\
&\qquad + \frac{6(\tau q-q^2)((\tau -2q) \nabla q + q\nabla\tau)}{(1+q^2-2q\tau)^{5/2}}
+ \frac{15(\tau q-q^2)^2((\tau - q) \nabla q + q\nabla\tau)}{(1+q^2-2q\tau)^{7/2}}\\
&= \frac{(\tau -4q) \nabla q + q\nabla\tau }{(1+q^2-2q\tau)^{3/2}} \\
&\qquad \qquad + \frac{3(\tau q-2q^2)((\tau - q) \nabla q + q\nabla\tau) + 6(\tau q-q^2)((\tau -2q) \nabla q + q\nabla\tau)}{(1+q^2-2q\tau)^{5/2}}\\
&\qquad\qquad + \frac{15(\tau q-q^2)^2((\tau - q) \nabla q + q\nabla\tau)}{(1+q^2-2q\tau)^{7/2}}
\label{eq:nablaqdq2phi}\\\notag\\
&\eqqcolon \frac{\mathrm{ID4}}{\mathrm{denom3}} + \frac{3\mathrm{tqm2qsq*ID} + 6\mathrm{tqmqsq*ID2}}{\mathrm{denom5}} + \frac{15\mathrm{tqmqsq2*ID}}{\mathrm{denom7}}\\ \ \\
&\nabla \left[\left(q\intd_q\right)^3 \phi(q)\right]\\
&= \nabla\left(\frac{\tau q-4q^2}{(1+q^2-2q\tau)^{3/2}} + \frac{9(\tau q-q^2)(\tau q-2q^2)}{(1+q^2-2q\tau)^{5/2}} + \frac{15(\tau q-q^2)^3}{(1+q^2-2q\tau)^{7/2}}\right)\\
&= \frac{q\nabla\tau + \tau \nabla q - 8q\nabla q}{(1+q^2-2q\tau)^{3/2}} 
- \frac{3}{2}\frac{(\tau q-4q^2)(2q\nabla q - 2\tau\nabla q - 2q \nabla \tau)}{(1+q^2-2q\tau)^{5/2}} \\
&\qquad + \frac{9(q\nabla\tau + \tau \nabla q - 2q\nabla q)(\tau q-2q^2) + 9(\tau q-q^2)(q\nabla\tau + \tau\nabla q - 4q\nabla q)}{(1+q^2-2q\tau)^{5/2}}\\ 
&\qquad \qquad - \frac{5}{2}\frac{9(\tau q-q^2)(\tau q-2q^2)(2q\nabla q - 2\tau\nabla q - 2q \nabla \tau)}{(1+q^2-2q\tau)^{7/2}}\\
&\qquad + \frac{45(\tau q-q^2)^2(q\nabla \tau +\tau\nabla q - 2q\nabla q)}{(1+q^2-2q\tau)^{7/2}}
- \frac{7}{2}\frac{15(\tau q-q^2)^3(2q\nabla q - 2\tau\nabla q - 2q \nabla \tau)}{(1+q^2-2q\tau)^{9/2}}\\
&= \frac{(\tau - 8q)\nabla q + q\nabla\tau}{(1+q^2-2q\tau)^{3/2}} 
+ \frac{3(\tau q-4q^2)((\tau - q)\nabla q + q\nabla\tau)}{(1+q^2-2q\tau)^{5/2}} \\
&\qquad + \frac{9((\tau - 2q)\nabla q + q\nabla\tau)(\tau q-2q^2) + 9(\tau q-q^2)((\tau - 4q)\nabla q + q\nabla\tau)}{(1+q^2-2q\tau)^{5/2}}\\ 
&\qquad \qquad + \frac{45(\tau q-q^2)(\tau q-2q^2)((\tau - q)\nabla q + q\nabla\tau)}{(1+q^2-2q\tau)^{7/2}}\\
&\qquad + \frac{45(\tau q-q^2)^2((\tau - 2q)\nabla q + q\nabla\tau)}{(1+q^2-2q\tau)^{7/2}}
+ \frac{105(\tau q-q^2)^3((\tau - q)\nabla q + q\nabla\tau)}{(1+q^2-2q\tau)^{9/2}}\\
&= \frac{(\tau - 8q)\nabla q + q\nabla\tau}{(1+q^2-2q\tau)^{3/2}} 
+ \frac{3(\tau q-4q^2)((\tau - q)\nabla q + q\nabla\tau)}{(1+q^2-2q\tau)^{5/2}}\\ 
&\qquad + \frac{9((\tau - 2q)\nabla q + q\nabla\tau)(\tau q-2q^2) + 9(\tau q-q^2)((\tau - 4q)\nabla q + q\nabla\tau)}{(1+q^2-2q\tau)^{5/2}}\\ 
& \qquad + \frac{45(\tau q-q^2)(\tau q-2q^2)((\tau - q)\nabla q + q\nabla\tau) + 45(\tau q-q^2)^2((\tau - 2q)\nabla q + q\nabla\tau)}{(1+q^2-2q\tau)^{7/2}}\\
&\qquad+ \frac{105(\tau q-q^2)^3((\tau - q)\nabla q + q\nabla\tau)}{(1+q^2-2q\tau)^{9/2}}
\label{eq:nablaqdq3phi}\\\notag\\
&\eqqcolon \frac{\mathrm{ID8}}{\mathrm{denom3}} 
+ \frac{3\mathrm{tqm4qsq*ID} + 9\mathrm{ID2*tqm2qsq} + 9\mathrm{tqmqsq*ID4}}{\mathrm{denom5}}\\ 
&\qquad+ \frac{45\mathrm{tqmqsq*tqm2qsq*ID} + 45\mathrm{tqmqsq2*ID2}}{\mathrm{denom7}}
+ \frac{105\mathrm{tqmqsq3*ID}}{\mathrm{denom9}}
\end{align}
\begin{align}
&\nabla \left[\left(q\intd_q\right)^4 \phi(q)\right]\\
&= \nabla\left( \frac{\tau q-8q^2}{(1+q^2-2q\tau)^{3/2}} + \frac{12(\tau q-4q^2)(\tau q-q^2) + 9(\tau q-2q^2)^2}{(1+q^2-2q\tau)^{5/2}} \right. \\
&\qquad + \left. \frac{90(\tau q-q^2)^2(\tau q-2q^2)}{(1+q^2-2q\tau)^{7/2}} + \frac{105(\tau q-q^2)^4}{(1+q^2-2q\tau)^{9/2}}\right)\\
&= 
\frac{q\nabla\tau + \tau \nabla q -16q\nabla q}{(1+q^2-2q\tau)^{3/2}} 
-\frac{3}{2}\frac{(\tau q-8q^2)(2q\nabla q - 2\tau\nabla q - 2q \nabla \tau)}{(1+q^2-2q\tau)^{5/2}}\\ 
&\qquad + \frac{12(q\nabla\tau + \tau\nabla q - 8q\nabla q)(\tau q-q^2)
+ 12(\tau q-4q^2)(q\nabla\tau + \tau\nabla q-2q\nabla q)}{(1+q^2-2q\tau)^{5/2}} \\
&\qquad+ \frac{18(\tau q-2q^2)(q\nabla\tau + \tau\nabla q - 4q\nabla q)}{(1+q^2-2q\tau)^{5/2}} \\
&\qquad - \frac{5}{2}\frac{(12(\tau q-4q^2)(\tau q-q^2) + 9(\tau q-2q^2)^2)(2q\nabla q - 2\tau\nabla q - 2q \nabla \tau)}{(1+q^2-2q\tau)^{7/2}} \\
&\qquad + \frac{180(\tau q-q^2)(q\nabla \tau + \tau\nabla q - 2q\nabla q)(\tau q-2q^2)}{(1+q^2-2q\tau)^{7/2}} \\
&\qquad + \frac{90(\tau q-q^2)^2(q\nabla\tau+\tau \nabla q-4q\nabla q)}{(1+q^2-2q\tau)^{7/2}} \\
&\qquad -\frac{7}{2}\frac{90(\tau q-q^2)^2(\tau q-2q^2)(2q\nabla q - 2\tau\nabla q - 2q \nabla \tau)}{(1+q^2-2q\tau)^{9/2}} \\
&\qquad + \frac{420(\tau q-q^2)^3(q\nabla\tau+\tau\nabla q - 2q\nabla q)}{(1+q^2-2q\tau)^{9/2}}\\
&\qquad -\frac{9}{2}\frac{105(\tau q-q^2)^4(2q\nabla q - 2\tau\nabla q - 2q \nabla \tau)}{(1+q^2-2q\tau)^{11/2}}\\
&=\frac{(\tau- 16q)\nabla q + q \nabla \tau}{(1+q^2-2q\tau)^{3/2}} \\
&\qquad +\frac{3(\tau q-8q^2)((\tau- q)\nabla q + q \nabla \tau) + 12((\tau- 8q)\nabla q + q \nabla \tau)(\tau q-q^2)}{(1+q^2-2q\tau)^{5/2}}\\
&\qquad + \frac{12(\tau q-4q^2)((\tau- 2q)\nabla q + q \nabla \tau) + 18(\tau q-2q^2)((\tau- 4q)\nabla q + q \nabla \tau)}{(1+q^2-2q\tau)^{5/2}} \\
&\qquad +\frac{60(\tau q-4q^2)(\tau q-q^2) ((\tau- q)\nabla q + q \nabla \tau)
+ 45(\tau q-2q^2)^2)((\tau- q)\nabla q + q \nabla \tau)}{(1+q^2-2q\tau)^{7/2}} \\
&\qquad + \frac{180(\tau q-q^2)((\tau- 2q)\nabla q + q \nabla \tau)(\tau q-2q^2) + 90(\tau q-q^2)^2((\tau- 4q)\nabla q + q \nabla \tau)}{(1+q^2-2q\tau)^{7/2}} \\
&\qquad +\frac{630(\tau q-q^2)^2(\tau q-2q^2)((\tau- q)\nabla q + q \nabla \tau) + 420(\tau q-q^2)^3((\tau- 2q)\nabla q + q \nabla \tau)}{(1+q^2-2q\tau)^{9/2}}\\
&\qquad +\frac{945(\tau q-q^2)^4((\tau- q)\nabla q + q \nabla \tau)}{(1+q^2-2q\tau)^{11/2}}
\label{eq:nablaqdq4phi}
\end{align}
\begin{align}
&=\frac{\mathrm{ID16}}{\mathrm{denom3}}
+ \frac{3\mathrm{tqm8qsq*ID} + 12\mathrm{ID8*tqmqsq} + 12\mathrm{tqm4qsq*ID2} + 18\mathrm{tqm2qsq*ID4}}{\mathrm{denom5}} \\
&\qquad +\frac{60\mathrm{tqm4qsq*tqmqsq*ID} + 45\mathrm{tqm2qsq2*ID}}{\mathrm{denom7}} \\
&\qquad + \frac{180\mathrm{tqmqsq*ID2*tqm2qsq} + 90\mathrm{tqmqsq2*ID4}}{\mathrm{denom7}} \\
&\qquad +\frac{630\mathrm{tqmqsq2*tqm2qsq*ID} + 420\mathrm{tqmqsq3*ID2}}{\mathrm{denom9}}
+\frac{945\mathrm{tqmqsq4*ID}}{\mathrm{denom11}}\\\notag \\
&\nabla \left[ \left(q\intd_q\right)^5 \phi(q)\right]\\
&= \nabla\left( \frac{\tau q-16q^2}{(1+q^2-2q\tau)^{3/2}} \right.\\
&\qquad +\frac{15(\tau q-8q^2)(\tau q-q^2) + 12(\tau q-4q^2)(\tau q-2q^2) + 18(\tau q-2q^2)(\tau q-4q^2)}{(1+q^2-2q\tau)^{5/2}} \\
&\qquad + \frac{150(\tau q-4q^2)(\tau q-q^2)^2 + 225(\tau q-2q^2)^2(\tau q-q^2)}{(1+q^2-2q\tau)^{7/2}}\\
&\qquad + \frac{1050(\tau q-q^2)^3(\tau q-2q^2)}{(1+q^2-2q\tau)^{9/2}} 
+ \left. \frac{945(\tau q-q^2)^5}{(1+q^2-2q\tau)^{11/2}} \right)\\
&= \frac{q\nabla\tau + \tau\nabla q -32q\nabla q}{(1+q^2-2q\tau)^{3/2}} 
- \frac{3}{2}\frac{(\tau q-16q^2)(2q\nabla q - 2\tau\nabla q - 2q \nabla \tau)}{(1+q^2-2q\tau)^{5/2}}\\
&\qquad + \frac{15(q\nabla\tau + \tau\nabla q -16q\nabla q)(\tau q-q^2) + 15(\tau q-8q^2)(q\nabla\tau + \tau\nabla q-2q\nabla q)}{(1+q^2-2q\tau)^{5/2}} \\
&\qquad + \frac{12(q\nabla\tau + \tau\nabla q -8q\nabla q)(\tau q-2q^2)+12(\tau q-4q^2)(q\nabla\tau + \tau\nabla q -4q\nabla q)}{(1+q^2-2q\tau)^{5/2}} \\
&\qquad + \frac{18(q\nabla\tau + \tau\nabla q -4q\nabla q)(\tau q-4q^2)+18(\tau q-2q^2)(q\nabla\tau + \tau\nabla q -8q\nabla q)}{(1+q^2-2q\tau)^{5/2}} \\
&\qquad -\frac{5}{2}\frac{15(\tau q-8q^2)(\tau q-q^2)(2q\nabla q - 2\tau\nabla q - 2q \nabla \tau)}{(1+q^2-2q\tau)^{7/2}}\\
&\qquad -\frac{5}{2}\frac{12(\tau q-4q^2)(\tau q-2q^2)(2q\nabla q - 2\tau\nabla q - 2q \nabla \tau)}{(1+q^2-2q\tau)^{7/2}}\\
&\qquad -\frac{5}{2}\frac{18(\tau q-2q^2)(\tau q-4q^2)(2q\nabla q - 2\tau\nabla q - 2q \nabla \tau)}{(1+q^2-2q\tau)^{7/2}}\\
&\qquad +\frac{150(q\nabla\tau + \tau\nabla q -8q\nabla q)(\tau q-q^2)^2 + 300(\tau q-4q^2)(\tau q-q^2)(q\nabla\tau + \tau\nabla q -2q\nabla q) }{(1+q^2-2q\tau)^{7/2}}\\
&\qquad + \frac{450(\tau q-2q^2)(q\nabla\tau + \tau\nabla q -4q\nabla q)(\tau q-q^2) + 225(\tau q-2q^2)^2(q\nabla\tau + \tau\nabla q -2q\nabla q)}{(1+q^2-2q\tau)^{7/2}}\\
&\qquad -\frac{7}{2}\frac{150(\tau q-4q^2)(\tau q-q^2)^2(2q\nabla q - 2\tau\nabla q - 2q \nabla \tau)}{(1+q^2-2q\tau)^{9/2}}\\ 
&\qquad - \frac{7}{2}\frac{225(\tau q-2q^2)^2(\tau q-q^2)(2q\nabla q - 2\tau\nabla q - 2q \nabla \tau)}{(1+q^2-2q\tau)^{9/2}}\\
&\qquad +\frac{3150(\tau q-q^2)^2(q\nabla\tau + \tau\nabla q -2q\nabla q)(\tau q-2q^2) + 1050(\tau q-q^2)^3(q\nabla\tau + \tau\nabla q -4q\nabla q)}{(1+q^2-2q\tau)^{9/2}} \\
&\qquad -\frac{9}{2}\frac{1050(\tau q-q^2)^3(\tau q-2q^2)(2q\nabla q - 2\tau\nabla q - 2q \nabla \tau)}{(1+q^2-2q\tau)^{11/2}} \\
&\qquad + \frac{4725(\tau q-q^2)^4(q\nabla\tau + \tau\nabla q -2q\nabla q)}{(1+q^2-2q\tau)^{11/2}} 
-\frac{11}{2}\frac{945(\tau q-q^2)^5(2q\nabla q - 2\tau\nabla q - 2q \nabla \tau)}{(1+q^2-2q\tau)^{13/2}} 
\end{align}
\begin{align}
&= \frac{q\nabla\tau + \tau\nabla q -32q\nabla q}{(1+q^2-2q\tau)^{3/2}} 
+ \frac{3(\tau q-16q^2)((\tau - q)\nabla q + q \nabla \tau)}{(1+q^2-2q\tau)^{5/2}}\\
&\qquad + \frac{15((\tau - 16q)\nabla q + q \nabla \tau)(\tau q-q^2) + 15(\tau q-8q^2)((\tau - 2q)\nabla q + q \nabla \tau)}{(1+q^2-2q\tau)^{5/2}} \\
&\qquad + \frac{12((\tau - 8q)\nabla q + q \nabla \tau)(\tau q-2q^2)+12(\tau q-4q^2)((\tau - 4q)\nabla q + q \nabla \tau)}{(1+q^2-2q\tau)^{5/2}} \\
&\qquad + \frac{18((\tau - 4q)\nabla q + q \nabla \tau)(\tau q-4q^2)+18(\tau q-2q^2)((\tau - 8q)\nabla q + q \nabla \tau)}{(1+q^2-2q\tau)^{5/2}} \\
&\qquad + \frac{75(\tau q-8q^2)(\tau q-q^2)((\tau - q)\nabla q + q \nabla \tau)}{(1+q^2-2q\tau)^{7/2}}\\
&\qquad +\frac{60(\tau q-4q^2)(\tau q-2q^2)((\tau - q)\nabla q + q \nabla \tau)}{(1+q^2-2q\tau)^{7/2}}\\
&\qquad +\frac{90(\tau q-2q^2)(\tau q-4q^2)((\tau - q)\nabla q + q \nabla \tau)}{(1+q^2-2q\tau)^{7/2}}\\
&\qquad +\frac{150((\tau - 8q)\nabla q + q \nabla \tau)(\tau q-q^2)^2 + 300(\tau q-4q^2)(\tau q-q^2)((\tau - 2q)\nabla q + q \nabla \tau) }{(1+q^2-2q\tau)^{7/2}}\\
&\qquad + \frac{450(\tau q-2q^2)((\tau - 4q)\nabla q + q \nabla \tau)(\tau q-q^2) + 225(\tau q-2q^2)^2((\tau - 2q)\nabla q + q \nabla \tau)}{(1+q^2-2q\tau)^{7/2}}\\
&\qquad + \frac{1050(\tau q-4q^2)(\tau q-q^2)^2((\tau - q)\nabla q + q \nabla \tau)}{(1+q^2-2q\tau)^{9/2}}\\ 
&\qquad + \frac{1575(\tau q-2q^2)^2(\tau q-q^2)((\tau - q)\nabla q + q \nabla \tau)}{(1+q^2-2q\tau)^{9/2}}\\
&\qquad +\frac{3150(\tau q-q^2)^2((\tau - 2q)\nabla q + q \nabla \tau)(\tau q-2q^2) + 1050(\tau q-q^2)^3((\tau - 4q)\nabla q + q \nabla \tau)}{(1+q^2-2q\tau)^{9/2}} \\
&\qquad +\frac{9450(\tau q-q^2)^3(\tau q-2q^2)((\tau - q)\nabla q + q \nabla \tau)}{(1+q^2-2q\tau)^{11/2}} \\
&\qquad + \frac{4725(\tau q-q^2)^4((\tau - 2q)\nabla q + q \nabla \tau)}{(1+q^2-2q\tau)^{11/2}} 
+\frac{10395(\tau q-q^2)^5((\tau - q)\nabla q + q \nabla \tau)}{(1+q^2-2q\tau)^{13/2}} 
\end{align}
\begin{align}
&= \frac{(\tau-32q)\nabla q + q\nabla \tau}{(1+q^2-2q\tau)^{3/2}} \\
&\qquad + \frac{3(\tau q-16q^2)((\tau - q)\nabla q + q \nabla \tau)+ 15((\tau - 16q)\nabla q + q \nabla \tau)(\tau q-q^2)}{(1+q^2-2q\tau)^{5/2}} \\
&\qquad + \frac{15(\tau q-8q^2)((\tau - 2q)\nabla q + q \nabla \tau) + 30((\tau - 8q)\nabla q + q \nabla \tau)(\tau q-2q^2)}{(1+q^2-2q\tau)^{5/2}} \\
&\qquad + \frac{30(\tau q-4q^2)((\tau - 4q)\nabla q + q \nabla \tau)}{(1+q^2-2q\tau)^{5/2}} \\
&\qquad + \frac{75(\tau q-8q^2)(\tau q-q^2)((\tau - q)\nabla q + q \nabla \tau) + 150(\tau q-4q^2)(\tau q-2q^2)((\tau - q)\nabla q + q \nabla \tau)}{(1+q^2-2q\tau)^{7/2}}\\
&\qquad + \frac{150((\tau - 8q)\nabla q + q \nabla \tau)(\tau q-q^2)^2 + 300(\tau q-4q^2)(\tau q-q^2)((\tau - 2q)\nabla q + q \nabla \tau)}{(1+q^2-2q\tau)^{7/2}}\\
&\qquad + \frac{450(\tau q-2q^2)((\tau - 4q)\nabla q + q \nabla \tau)(\tau q-q^2) + 225(\tau q-2q^2)^2((\tau - 2q)\nabla q + q \nabla \tau)}{(1+q^2-2q\tau)^{7/2}}\\
&\qquad + \frac{1050(\tau q-4q^2)(\tau q-q^2)^2((\tau - q)\nabla q + q \nabla \tau) + 1575(\tau q-2q^2)^2(\tau q-q^2)((\tau - q)\nabla q + q \nabla \tau)}{(1+q^2-2q\tau)^{9/2}}\\
&\qquad + \frac{3150(\tau q-q^2)^2((\tau - 2q)\nabla q + q \nabla \tau)(\tau q-2q^2) + 1050(\tau q-q^2)^3((\tau - 4q)\nabla q + q \nabla \tau)}{(1+q^2-2q\tau)^{9/2}} \\
&\qquad +\frac{9450(\tau q-q^2)^3(\tau q-2q^2)((\tau - q)\nabla q + q \nabla \tau) + 4725(\tau q-q^2)^4((\tau - 2q)\nabla q + q \nabla \tau)}{(1+q^2-2q\tau)^{11/2}} \\
&\qquad + \frac{10395(\tau q-q^2)^5((\tau - q)\nabla q + q \nabla \tau)}{(1+q^2-2q\tau)^{13/2}} 
\label{eq:nablaqdq5phi}\\ \notag\\
&\eqqcolon \frac{\mathrm{ID32}}{\mathrm{denom3}} \\
&\qquad + \frac{3\mathrm{tqm16qsq*ID} + 15\mathrm{ID16*tqmqsq}}{\mathrm{denom5}} \\
&\qquad + \frac{15\mathrm{tqm8qsq*ID2} + 30\mathrm{ID8*tqm2qsq}}{\mathrm{denom5}} \\
&\qquad + \frac{30\mathrm{tqm4qsq*ID4}}{\mathrm{denom5}} \\
&\qquad + \frac{75\mathrm{tqm8qsq*tqmqsq*ID} + 150\mathrm{tqm4qsq*tqm2qsq*ID}}{\mathrm{denom7}}\\
&\qquad + \frac{150\mathrm{ID8*tqmqsq2} + 300\mathrm{tqm4qsq*tqmqsq*ID2}}{\mathrm{denom7}}\\
&\qquad + \frac{450\mathrm{tqm2qsq*ID4*tqmqsq} + 225\mathrm{tqm2qsq2*ID2}}{\mathrm{denom7}}\\
&\qquad + \frac{1050\mathrm{tqm4qsq*tqmqsq2*ID} + 1575\mathrm{tqm2qsq2*tqmqsq*ID}}{\mathrm{denom9}}\\
&\qquad + \frac{3150\mathrm{tqmqsq2*ID2*tqm2qsq} + 1050\mathrm{tqmqsq3*ID4}}{\mathrm{denom9}} \\
&\qquad +\frac{9450\mathrm{tqmqsq3*tqm2qsq*ID} + 4725\mathrm{tqmqsq4*ID2}}{\mathrm{denom11}} \\
&\qquad + \frac{10395\mathrm{tqmqsq5*ID}}{\mathrm{denom13}} 
\end{align}

\end{document}